\documentclass[11pt,reqno]{amsart}
\usepackage{amsmath, amsthm, amssymb}
\usepackage{amsfonts}
\usepackage[ansinew]{inputenc}
\usepackage[dvips]{epsfig}
\usepackage{graphicx}
\usepackage[english]{babel}
\usepackage{bbm}

\usepackage{cite}
\usepackage{graphicx}
\usepackage{amscd}
\usepackage{bm}
\usepackage{enumerate}

\usepackage{verbatim}
\usepackage{hyperref}
\usepackage{amstext}
\usepackage{latexsym}
\theoremstyle{plain}
\newtheorem{theorem}{Theorem}[section]

\newtheorem{corollary}[theorem]{Corollary}
\newtheorem{proposition}[theorem]{Proposition}
\newtheorem{lemma}[theorem]{Lemma}
\newtheorem{remark}[theorem]{Remark}

\newtheorem{assumptions}[theorem]{Assumptions}
\numberwithin{theorem}{section}
\numberwithin{equation}{section}

\newcommand{\average}{{\mathchoice {\kern1ex\vcenter{\hrule height.4pt
				width 6pt depth0pt} \kern-9.7pt} {\kern1ex\vcenter{\hrule
				height.4pt width 4.3pt depth0pt} \kern-7pt} {} {} }}

\def\R{\mathbb{R}}

\renewcommand{\a }{\alpha }
\renewcommand{\b }{\beta }
\renewcommand{\d}{\delta }

\newcommand{\D }{\Delta }

\newcommand{\e }{\varepsilon }
\newcommand{\g }{\gamma}

\newcommand{\G }{\Gamma}
\renewcommand{\l }{\lambda }

\newcommand{\n }{\nabla }

\newcommand{\s }{\sigma }

\renewcommand{\t }{\tau }

\renewcommand{\th }{\theta }

\newcommand{\ov}{\overline}

\newcommand{\be}{\begin{equation}}
\newcommand{\ee}{\end{equation}}

\newcommand{\de}{\partial}

\newcommand{\ti}{\widetilde}

\renewcommand{\k}{\kappa}

\usepackage{mathrsfs}

\newcommand{\calH }{\mathcal{H}}
\newcommand{\calO }{\mathcal{O}}

\newcommand{\calE }{\mathcal{E}}
\newcommand{\calL }{\mathcal{L}}

\newcommand{\calD }{\mathcal{D}}

\newcommand{\calB }{\mathcal{B}}

\newcommand{\calS}{{\mathcal S}}

\newcommand{\mbL }{\mathbb{L}}

\newcommand{\N}{\mathbb{N}}

\newcommand{\cA}{{\mathcal A}}
\newcommand{\cB}{{\mathcal B}}
\newcommand{\cC}{{\mathcal C}}

\newcommand{\cG}{{\mathcal G}}
\newcommand{\cH}{{\mathcal H}}

\newcommand{\cK}{{\mathcal K}}
\newcommand{\cL}{{\mathcal L}}

\newcommand{\cO}{{\mathcal O}}

\newcommand{\cX}{{\mathcal X}}
\newcommand{\cY}{{\mathcal Y}}

\DeclareMathOperator{\id}{id}

\renewcommand{\epsilon}{\varepsilon}

\begin{document}
	\title[Short time existence and  smoothness  of the nonlocal mean curvature flow]{Short time existence  and smoothness of the nonlocal mean curvature flow of graphs}
	
	\author{Anoumou Attiogbe}
	\address{A. A. : African Institute for Mathematical Sciences in Senegal, KM 2, Route de
		Joal, B.P. 14 18. Mbour, Senegal.}
	\email{anoumou.attiogbe@aims-senegal.org}
	
	\author{Mouhamed Moustapha Fall}
	\address{M. M. F.: African Institute for Mathematical Sciences in Senegal, KM 2, Route de
		Joal, B.P. 14 18. Mbour, Senegal.}
	\email{mouhamed.m.fall@aims-senegal.org}
	
	
	\author{Tobias Weth}
	\address{T. W.: Institut f\"ur Mathematik,  Goethe-Universit\"at,  Frankfurt,  Robert-Mayer-Str. 10,  D-60629 Frankfurt}
	\email{weth@math.uni-frankfurt.de}
	\thanks{Keywords: Nonlocal mean curvature flow; Quasilinear evolution equations; Analytic semigroup theory;    Fixed point theorem.} 
	\maketitle

	\begin{abstract}
	We consider the  geometric evolution problem of entire graphs moving by  fractional mean curvature. For this, we study the associated nonlocal quasilinear evolution equation satisfied by the family of graph functions. We establish, using an analytic semigroup approach, short time existence, uniqueness and optimal H\"older regularity in time and space of classical solutions of the nonlocal equation, depending on the regularity of the initial graph. The method also yields $C^\infty-$smoothness estimates of the evolving graphs for positive times.
	\end{abstract}


	\section{Introduction}
In the present paper, we study the geometric evolution problem of hypersurfaces moving by fractional mean curvature. More precisely we consider the problem of existence and uniqueness of a family of (sufficiently regular) open subsets $\{E(t)\}_{t>0}$ of $\R^N$ satisfying
	\begin{align}\label{GeneralFormulation}
		\partial_tX_t\cdot\nu(X_t):=-H_{E(t)}^\alpha(X_t),\;\;\;\mbox{for all $X_t\in \partial E(t)$ and }t\in[0,T],
	\end{align}
where,  	$\nu$ is the unit normal vector field on $\partial E(t)$ and $H_{E(t)}^\alpha$ is the fractional mean curvature of order $\alpha\in(0,1)$ at $X_t\in \partial E(t)$.
%
Recall that for a set $E\subset \R^N$ of class $C^{1+\b}$, with  $\beta>\a  $,  the fractional (nonlocal) mean curvature of $\partial E$  is well defined and it is given  at a point $x\in \de E$ by 
\begin{align}\label{Definition01}
H_{E}^\alpha(x):= P. V. \int_{\R^N} \frac{\mathbbm{1}_{E^c}(y)-\mathbbm{1}_{E}(y)}{|x-y|^{N+\alpha}} dy=\lim_{\e \to 0} \int_{\R^N \setminus B_\e(x)} \frac{\mathbbm{1}_{E^c}(y)-\mathbbm{1}_{E}(y) }{|x-y|^{N+\alpha}} dy,
\end{align}
where $\mathbbm{1}_{E}$ denotes the characteristic function of the set  $E$ and ${E^c}:=\mathbb{R}^{N}\setminus E$. 
We recall that the notion of fractional (nonlocal) mean curvature of order $\alpha\in(0,1)$ appeared for the first time in the work of Caffarelli and Souganidis in \cite{caffarelli2010convergence} around 2008. Since that time, the nonlocal mean curvature  has been studied extensively in various settings, see e.g. \cite{abatangelo2014notion,bucur2016nonlocal, saez2019evolution,fall2018constant} and the references therein.

As a nonlocal conterpart of the classical mean curvature flow (see e.g. \cite{Huisken1,Mantegazza} and the references therein), the nonlocal mean curvature flow has attracted much interest in recent years, see e.g \cite{CSV,cameron2019eventual,cesaroni2020symmetric,cesaroni2021fractional,cintifractional,imbert2009level,julin2020short,saez2019evolution, attiogbe2021nonlocal} and the survey paper \cite{cintifractional}. 
  The first paper dealing with the existence and uniqueness of solutions of this flow  is \cite{imbert2009level}, where Imbert  used the level set formulation of the geometric flow to prove   existence and uniqueness  of   viscosity solutions  while the recent  papers \cite{cameron2019eventual} and   \cite{cesaroni2021fractional} deal with  regularity results  of viscosity solution of  the flow starting from  Lipschitz graphs. We also mention the paper \cite{saez2019evolution} where the autors proved   a
 comparison principle  for the flow \eqref{GeneralFormulation} and   found bounds on
 the maximal existence time and uniqueness of smooth solutions.

Despite the attention that the nonlocal mean curvature flow has already received, the existence and regularity of classical solutions remained an open problem until the recent work \cite{julin2020short}. In this paper, the authors prove short time existence and uniqueness of  classical solution  to the fractional mean curvature flow, starting from  a  bounded   $C^{1,1}$ initial set.

In the present paper,  we consider the complementary case of short time existence and uniqueness of classical solutions of problem \eqref{GeneralFormulation} when the sets   $E(t)$ are given by subgraphs of functions    $u(t,\cdot)\in C^{1+\b}_{loc}(\R^{N-1})$, where $\b>\alpha$. 
We recall first  that if  $ E_u:=\{(x,y)\in\mathbb{R}^{N-1}\times\mathbb{R}:y<u(x)\}$ then, see e.g. \cite[Proposition 3.5]{fall2018constant},  we have that   
\begin{align}
H(u)(x): =H^\a_{E_u}(x,u(x))= P.V.\int_{\R^{N-1}}\frac{\cG(p_u({x},{y})) }{|{x}-{y}|^{N-1+\a}} d{y},
\end{align}
where 
 \be \label{eq:P_u-G}
  p_u({x},{y})= \frac{u({y})-u({x})}{|{x}-{y}|} ,
\qquad
\cG(p):=-\int_{-p}^{p}\frac{d\t}{(1+\t^2)^{\frac{N+\a}{2}}}.
\ee
We point out that, due to the boundedness of the function $\cG$, the expression $H(u)(x)$ is well-defined if $u: \R^{N-1} \to \R$ is measurable and of class $C^{1+\beta}$ for some $\beta>\alpha$ in a neighborhood of $x$. With this expression of the fractional mean curvature, we can derive the quasilinear evolution problem corresponding to   \eqref{GeneralFormulation} when $E_u(t):=\{(x(t),y(t))\in\mathbb{R}^{N-1}\times\mathbb{R}:y(t)<u(t,x(t))\}$,  with  
 $u: [0, T]\times  \mathbb{R}^{N-1}\to \R$.  Indeed, recall that the unit interior normal vector field on $\partial E_u(t)$ is given by $\nu(X_t):=\frac{(-\nabla u(t,x(t)),1) }{\sqrt{1+|\nabla u(t,x(t))|^2}}$. Therefore for every $	X_t\in\partial E_u(t)=\{(x(t),u(t,x(t))):x(t)\in\mathbb{R}^{N-1}\}$,  we have that
\begin{align*}
	\partial_tX_t=\Big(\dot{x}(t),\dot{x}(t)\cdot\nabla u(t,x(t))+\partial_tu(t,x(t))\Big).
\end{align*}
Hence 
\begin{align}\label{NormalVelocity}
	\partial_tX_t\cdot\nu(X_t)&=\frac{-\dot{x}(t)\cdot\nabla u(t,x(t))}{\sqrt{1+|\nabla u(t,x(t))|^2}}+\frac{\dot{x}(t)\cdot\nabla u(t,x(t))}{\sqrt{1+|\nabla u(t,x(t))|^2}}+\frac{\partial_tu(t,x(t))}{\sqrt{1+|\nabla u(t,x(t))|^2}}\nonumber\\&=\frac{\partial_tu(t,x(t))}{\sqrt{1+|\nabla u(t,x(t))|^2}}.
\end{align}
Therefore, from \eqref{NormalVelocity} and \eqref{GeneralFormulation}, the evolution of $u$  is given by  
the flow associated to the quasilinear evolution  equation 
\begin{align}\label{QuasiLinear01}
\de_t u= -{\sqrt{1+|\nabla u|^2}} H(u),\;\;\;t\in(0,T],\;\;\;\;\;u(0)=u_0.
\end{align}

 Our first main result is  the following.
\begin{theorem}
\label{MainResult1}
	Let  $\nu>0$, $\b\in (\a,1)$,     $\rho\in (0,\frac{1}{1+\a})$ and  $\g_\rho:=\b+\rho(1+\a)$.  Then  for all  $u_0\in C^{1+\g_\rho}_{loc}(\mathbb{R}^{N-1})$,  with $\|\n u_0\|_{C^{\g_\rho}(\R^{N-1})}\leq \nu$,    there  exist  $T,C_0>0$ only depending on $\rho,\a,\b,\g,N$ and $\nu$ such that the problem
	\begin{align}
	\label{eq:main-eq-s1-cor}
	\left\{
	\begin{array}{rll}
	\de_tu+{\sqrt{1+|\nabla u|^2}} H(u)&=0&\hspace{.5cm}\mbox{in }\;[0,T]\times\mathbb{R}^{N-1}	\\
	u(0)&=u_0&\hspace{.5cm}\mbox{in }\;\mathbb{R}^{N-1}
	\end{array}
	\right.
	\end{align}
admits a  unique    solution $u\in C^{\rho}([0,T],C^{1+\b}_{loc}(\R^{N-1}))\cap C^{1+\rho}([0,T], C^{\b-\a}_{loc} (\R^{N-1}))$ satisfying 
\be \label{eq:main-1-cor}
	  \|u-u_0 \|_{C^{\rho}([0,T],C^{1+\b} (\R^{N-1}))\cap C^{1+\rho}([0,T], C^{\b-\a}  (\R^{N-1}))}  \leq C_0  .
\ee
If, in addition, $\n u_0\in C^{1+\g_\rho}(\mathbb{R}^{N-1})$ then for all $\b'\in (\a,\b)$    there exists $C>0$ only depending on $\rho,\a,\b,\g,N,\nu,T$ and $\b'$   such that 
\be \label{eq:main-3-cor}
	 	  \|\n  u \|_{C^{\rho}([0,T],C^{1+ \b'}(\R^{N-1}))}  \leq C \|\n u_0\|_{ C^{1+\g_\rho}(\mathbb{R}^{N-1})}  .
\ee
\end{theorem}

The following result complements Theorem~\ref{MainResult1} by providing, under the same assumptions, smoothness estimates for positive times.

\begin{theorem}
\label{MainResult1-part-2}
Under the assumptions of Theorem~\ref{MainResult1}, we have $u(t,\cdot)\in C^\infty(\R^{N-1})$ for every $t\in (0,T]$.
Moreover, for every $\b'\in (\a,\b)$, $\rho\in (0,\frac{\a}{1+\a}]$ and  for all $k\in \N\setminus\{0\}$, there exists $C_k>0$ only depending on $\rho,\a,\b,\g,N,\nu,\b',T$ and $k$   such that 
\be \label{eq:main-2-cor}
	 	  \|t^k \n u \|_{C^{\rho}([0,T],C^{k+ \b'}(\R^{N-1}))}  \leq C_k  .
\ee
\end{theorem}
Related to Theorem~\ref{MainResult1-part-2}, we also mention \cite{cesaroni2021fractional}, where the authors state that  solutions to the fractional mean flow of graph, starting from a Lipschitz graph, are $C^\infty$-smooth for positive times provided $H(u_0)\in L^\infty(\R^{N-1})$.  However   \cite{cesaroni2021fractional} does not provide enough details that confirm this statement.  Indeed, the authors apply  a priori H\"older regularity  estimates from \cite{Serra,ScSi}   to the linearization of the first equation of  \eqref{eq:main-eq-s1-cor}. However
 the linearization of  the weighted nonlocal mean curvature operator $u\mapsto {\sqrt{1+|\nabla u|^2}} H(u)$ does not fall in the class of nonlocal operators considered in  \cite{Serra,ScSi}. \\ 

Our next result provides  universal estimates of the Lipschitz norm and the  mean curvature of the evolving graph in terms of the initial data.
\begin{theorem}\label{th:mt-univ-est}
Under the assumptions of Theorem~\ref{MainResult1}, we have  
$$
\| \n u\|_{L^\infty((0,T)\times\R^{N-1} )}\leq \| \n u_0\|_{L^\infty( \R^{N-1} )}
$$
  and 
$$
\| \de_tu \|_{L^\infty((0,T)\times\R^{N-1} )}\leq   \| {\sqrt{1+|\nabla u_0|^2}}  H( u_0)\|_{L^\infty( \R^{N-1} )}.
$$
If moreover $u_0\in L^\infty(\R^{N-1})$, then    $
\|  u\|_{L^\infty((0,T)\times\R^{N-1} )}\leq \|  u_0\|_{L^\infty( \R^{N-1} )}.
$
\end{theorem}

As we shall see below, the existence  in  Theorem \ref{MainResult1}   is a consequence of a more general existence result concerning problem \eqref{eq:main-eq-s1-cor}. Indeed, we   define  first the Banach space   
$$
\cC^\theta_0(\R^{N-1}) =\ov{C^\infty_c(\mathbb{R}^{N-1})}^{\|\cdot \|_{C^\theta(\mathbb{R}^{N-1})}} \qquad\textrm{ for $\th\in \R_+\setminus \N$},
$$ 
endowed with the norm of $C^\theta(\mathbb{R}^{N-1})$.  We then write \eqref{eq:main-eq-s1-cor} as  
 \begin{align}
	\label{eq:main-eq-s112}
	\left\{
	\begin{array}{rll}
	\de_tu+ \cL_0u &=F(u)  &\hspace{.5cm}\mbox{in }\;[0,T]	\times\mathbb{R}^{N-1}\\
	u(0)&=u_0&\hspace{.5cm}\mbox{in }\;\mathbb{R}^{N-1},
	\end{array}
	\right.
	\end{align}
        where $\cL_0:=D\cH(u_0): \cC^{1+\b}_0(\mathbb{R}^{N-1})\to \cC^{\b-\a}_0(\mathbb{R}^{N-1})$ is the derivative of the weighted mean curvature operator
        \begin{align}\label{Weighted:MCoperator}
        u \mapsto \cH(u):={\sqrt{1+|\nabla u|^2}} H(u)
        \end{align}
        at the initial condition $u_0$ and the nonlinear map $u \mapsto F(u) = -\cH(u)+ \cL_0 u$ has the property that $F(u)-F(u_0)$ has superlinear growth in $u-u_0$. Our strategy is now to show that $\cL_0$ generates a strongly continuous analytic semigroup on $\cC^{\b-\a}_0(\mathbb{R}^{N-1})$ and to apply a fixed point argument.  However,  we need precise estimates, as we want all the regularity estimated by the H\"older norms of $\n u_0$.

The  main result from which we partly derive Theorem \ref{MainResult1} is  the following 

\begin{theorem}\label{MainResult}
	Let $\b\in (\a,1)$,   $\nu>0$, $\rho\in (0,\frac{1}{1+\a})$ and  $\g_\rho:=\b+\rho(1+\a)$.	Then, there exists $T>0$,   depending only on $\rho,\a,\b,N$ and $\nu$, such that for all      $u_0\in \cC^{1+\b}_0(\mathbb{R}^{N-1})$,  with 
		$$
	\|\n u_0 \|_{C^{\g_\rho}(\R^{N-1})} \leq \nu,
	$$
%
	 the initial value problem 
	\begin{align}\label{CauchyProblem}
	\left\{
	\begin{array}{rll}
	\de_t u+{\sqrt{1+|\nabla u|^2}} H(u)&=0&\hspace{.5cm}\mbox{in }\;[0,T]\times\mathbb{R}^{N-1}	\\
	u(0)&=u_0&\hspace{.5cm}\mbox{in }\;\mathbb{R}^{N-1}
	\end{array}
	\right.
	\end{align}
	admits a unique solution $u\in  C^{1+\rho}([0,T], \cC^{\b-\a}_0(\mathbb{R}^{N-1}))\cap C^{\rho}([0,T],\cC^{1+\b}_0(\mathbb{R}^{N-1}))$.  Moreover,  there exists $C_0>0$,   depending only on $\rho,\a,\b,N$ and $\nu$, such that 
\be \label{eq:estim-th3}
  \| u-u_0\|_{C^{1+\rho}([0,T], C^{\b-\a}(\mathbb{R}^{N-1}))\cap C^{\rho}([0,T],C^{1+\b}(\mathbb{R}^{N-1}))}   \leq C_0.
\ee
\end{theorem}	
We note that, thanks to the regularity of $E_{u}(t)$ proved in Theorem~\ref{MainResult}, the uniqueness of the flow $(E_{u}(t))_{t\in [0,T]}$ also follows from \cite{saez2019evolution}, where the authors proved that \eqref{GeneralFormulation} admits a unique   classical solution.  
On the other hand, the proof of our  Theorem~\ref{MainResult} yields existence and uniqueness simultaneously.

Recall that, provided the linear operator  $\cL_0:=D\cH(u_0): \cC^{1+\b}_0(\mathbb{R}^{N-1}) \to \cC^{\b-\a}_0 (\mathbb{R}^{N-1})$ generates an analytic semigroup,  to get optimal $C^\rho([0,T],\cC^{1+\b}_0(\mathbb{R}^{N-1}))$-regularity in time for the problem \eqref{eq:main-eq-s112},  it is necessary and sufficient that  $u_0$ satisfies $F(u_0)-\cL_0u_0=-\cH(u_0)\in \mathcal{D}_{\cL_0}(\rho,\infty)$. Here $\mathcal{D}_{\cL_0}(\rho,\infty)$  is a real interpolation space between the space $\cC^{\b-\a}_0(\mathbb{R}^{N-1})$ and the domain $  \cC^{1+\b}_0(\mathbb{R}^{N-1})$. We refer the reader to    Lunardi \cite{lunardi2012analytic} and Sinestrari \cite{sinestrari1985abstract} on this type of optimal regularity. \\
We shall show here that  $\mathcal{D}_{\cL_0}(\rho,\infty)=C^{\b+\rho(1+\a)-\a}(\mathbb{R}^{N-1})\cap  \cC^{\b-\a}_0(\mathbb{R}^{N-1})$ for $\rho\not= \frac{1+\a-\b}{1+\a}$  (see Proposition \ref{prop:car-inerpol-space} below).   In addition, we will prove that $\cH(u_0) \in \calD_{\cL_0}(\rho,\infty)$ if and only if  $\n u_0\in C^{\b+\rho(1+\a)}(\R^{N-1})$.\\
%
%
%




		The proof of  Theorem \ref{MainResult} is given in Section \ref{Section5} below. Main steps in this proof consist in establishing suitable regularity properties of the fractional mean curvature operator $H$ in H\"older spaces and in showing that the  operator $ D\cH(u_0):\cC^{1+\b}_0(\mathbb{R}^{N-1}) \subset \cC^{\b-\a}_0(\mathbb{R}^{N-1}) \to \cC^{\b-\a}_0(\mathbb{R}^{N-1})$ is the generator of a strongly continuous analytic semigroup. We shall derive the generation of a strongly analytic semigroup by combining a result from \cite{abels2009cauchy} for a class of non-smooth integro-differential operators defined $\cC^{1+\b}_0(\mathbb{R}^{N-1})$ with H\"older estimates for integral operators with anisotropic kernels represented suitably in polar coordinates. To obtain existence and uniqueness for the non homogeneous  Cauchy problem \eqref{eq:main-eq-s112}, we also need a characterization of   the real interpolation space $\mathcal{D}_{\cL_0}(\rho,\infty)$, see Section \ref{ss:general-nonloc}.
Once Theorem~\ref{MainResult}  is established, we derive Theorem~\ref{MainResult1} from it by an approximation argument. The uniqueness in Theorem~\ref{MainResult1} and the universal estimates in Theorem \ref{th:mt-univ-est}  are based on a suitable version of the comparison principle, which we also prove. In fact, our comparison principle yields further interesting qualitative properties of the flow such as preservation of positivity of $u_0$  and $H(u_0)$. \\

The paper is organized as follows. In Section \ref{SectionNotation},  we introduce   the function spaces we will use in this paper and we prove some preliminary results. Section \ref{Section3}  is devoted to the strongly continuous analytic semigroup generated by a general class of nonlocal operators. Section\ref{Section2} deals with the regularity properties  satisfied by  the weighted  nonlocal mean curvature operator  $\cH(u)$  and its linearization. Finally,  in   Section \ref{Section5}, we prove our  main results.




	\bigskip
	
	\noindent
	\textbf{ Acknowledgments} \\
 This work is supported by the Alexander von Humboldt foundation and the    German Academic
Exchange Service (DAAD).  The authors would like to thank Sven Jarohs and Esther Cabezas-Rivas for many fruitful discussions. 
	\section{Notation and preliminary estimates}\label{SectionNotation}
	%
	%
	%
For $m\in\mathbb{N} $, we denote by   $C_b^m(\mathbb{R}^{N-1})$     the space of   $m$-times bounded continuously differentiable functions  endowed  with the norm
	\begin{align*}
	\|u\|_{C_b^m }=\sum_{k=0}^m	\|D^ku\|_{L^\infty }:=\sum_{k=0}^m\sup_{x\in \R^{N-1}}|D^ku(x)|.
	\end{align*}
	For $0<\gamma<1$, we consider  the  space  of  H\"older continuous functions  
	\begin{align}
	C^{\gamma} (\mathbb{R}^{N-1})=\bigg\{u\in C_b(\mathbb{R}^{N-1}):\,[u]_{C^{\gamma } }:=\sup_{\stackrel{x,y\in\mathbb{R}^{N-1}}{x\neq y}}\frac{|u(x)-u(y)|}{|x-y|^{\gamma}}<\infty\bigg\}
	\end{align}
	endowed with the norm
	\begin{align*}
	u \mapsto \|u\|_{C^{\gamma }  }=	\|u\|_{L^\infty  }+[u]_{C^{\gamma }  }.
	\end{align*}
	More generally, for $m \in \N$ and $0<\gamma <1$, we let
        $$
        C^{m+\gamma} (\mathbb{R}^{N-1})=\{u\in C_b^m(\mathbb{R}^{N-1}):D^mu\in C^{\gamma}(\mathbb{R}^{N-1})\},
        $$
        endowed with the norm
	$$
        u \mapsto
        \|u\|_{C^{m+\gamma}  }=\|u\|_{C_b^m }+[D^mu]_{C^{\gamma }  }.
        $$
        For late use, we recall the  interpolation inequality that, for all $\g\in (0,1)$, $\g'\in[0,\g)$ and   for all $\e>0$, there exists $C_\e>0$ such that 
\be\label{eq:interpol}
\|\n v\|_{C^{\g'}} \leq \e \|\n v \|_{C^\g}+ C_\e \| v\|_{L^{\infty}} \qquad\textrm{ for all $v\in C^{1,\g}(\R^{N-1})$.}
\ee
	Next, we define 
	\be\label{eq:defcalCs0}
\cC^{m+\g}_0(\mathbb{R}^{N-1}) :	=\ov{ C^\infty_c(\R^{N-1}) }^{\|\cdot\|_{C^{m+\g}}}
	\ee
	endowed with the  $C^{m+\g}$-norm. We then recall  the following result.
	\begin{proposition}\cite[Proposition 2.1]{abels2009cauchy}\label{prop:car-C_0}
	Let $\b,\g\in (0,1)$, with $\g>\b$. Let $u\in C^{\g}(\R^{N-1})$ be such that $\lim_{R\to \infty} \|f\|_{C^{\b}(\R^{N-1}\setminus B_R)}=0$. Then $f\in \cC^\b_0 (\mathbb{R}^{N-1})$.
	\end{proposition}

We need some more notation related to space of time dependent functions.  Let $\cX$ be a Banach space. For $T>0$, we consider the Banach space  $L^\infty([0,T],\cX)$, consisting  of  bounded    functions $u:[0,T]\rightarrow \cX.$   The spaces $L^\infty([0,T],\cX)$ is endowed with the  norm $$\|u\|_{L^\infty([0,T],\cX)}=\sup_{t\in [0,T]}\|u(t)\|_\cX.$$
	For  $\mu\in (0,1)$, we define  the Banach space
	$$C^{1+\mu}([0,T],\cX)=\left\{u\in C^1([0,T],\cX)\, :\,[\de_t u]_{C^\mu([0,T],\cX)}=\sup_{ 0\leq s<t\leq T}\frac{\|\de_t u(t)-\de_t u(s)\|_\cX}{|t-s|^\mu}<+\infty \right\},$$
	endowed with the norm
	$\|u\|_{C^{1+\mu}([0,T],\cX)}=\|u\|_{L^\infty([0,T],\cX)}+[\de_t u]_{C^\mu([0,T],\cX)}.$

\subsection{Estimates in polar coordinates}
\label{sec:estim-polar-coord}

In our study of linearizations of the nonlocal mean curvature operator, we will have to consider integral operators whose associated kernels have anisotropic singularities which can be resolved partially in polar coordinates. For this, the following basic estimates will be of key importance.

\begin{lemma}\label{SemigroupLemma001}
Let $\g\in (\a,1)$, $v\in C^{1+\gamma}_{loc}(\R^{N-1})$ be such that $\n v\in L^\infty(\R^{N-1}) $. Let  $\mu\in C^{\g-\a}(\R^{N-1}\times [0,\infty)\times S^{N-2})$ and  $\nu\in C^{\g}(\R^{N-1}\times [0,\infty)\times S^{N-2})$ with  $\nu (\cdot,0,\cdot)=0$. We define
$$
I_ev(x):=\int_{S^{N-2}}\int_0^\infty \delta_e v(x,r,\theta)  \mu (x,r,\theta)r^{-2-\a }drd\theta
$$
and 
$$
I_ov(x):=\int_{S^{N-2}}\int_0^\infty \delta_o v(x,r,\theta)  \nu (x,r,\theta)r^{-2-\a }drd\theta
$$
for $x \in \R^{N-1}$, where
\be \label{eq:def-deeo}
		\delta_ev(x,r,\theta):=\frac{1}{2}(2v(x)-v(x+r\theta)-v(x-r\theta)) \quad\textrm{ and }\quad	\delta_ov(x,r,\theta):= v(x)-v(x-r\theta).
\ee
Then the functions $I_ev$ and $I_ov$ satisfy the following estimates with constants $C$ independent of $\mu, \nu$ and $v$.
\begin{itemize}
\item[(i)]   We have 
\be \label{eq:Io}
\|I_ov\|_{C^{\g-\a}}\leq C \| \n v\|_{L^\infty}\|\nu \|_{ C^{\g}(\R^{N-1}\times [0,\infty)\times S^{N-2}) }
\ee
and if $\n v \in C^\g(\R^{N-1})$,  then  
\be \label{eq:Ie}
\|I_ev\|_{C^{\g-\a}}\leq C \|\n v\|_{C^{\g}}\|\mu \|_{ C^{\g-\a}(\R^{N-1}\times [0,\infty)\times S^{N-2}) }.
\ee

\item[(ii)] If $\n v \in C^\g(\R^{N-1})$ and  $\n v$ is compactly supported,  then
		\be\label{eq:tozero-}
	\lim_{R\to\infty}   \|I_ev\|_{C ^{\gamma-\alpha}(\R^{N-1}\setminus B_R) }=	\lim_{R\to\infty}   \|I_o v\|_{C ^{\gamma-\alpha}(\R^{N-1}\setminus B_R) }=0.
		\ee
\item[(iii)] If  $v\in \cC^{1+\b}_0(\mathbb{R}^{N-1}) $, with $\b\in (\a,\g)$,  then $I_ev, I_ov\in {\cC^{\b-\a}_0}(\mathbb{R}^{N-1}) $.  	
\end{itemize}
\end{lemma}	

\begin{proof}
We assume for simplicity that  $\|\nu\|_{ C^{\g}(\R^{N-1}\times [0,\infty)\times S^{N-2}) }\leq 1$.
We  first observe that
\be \label{eq:Ao-est1}
		|\nu(x,r,\th )|\leq  \min (1, r^{\gamma}), \qquad  |\d_o v(x,r,\th )|=r \left|\int_0^1 \n v(x-rt\th )\cdot \th dt \right|\leq r \|\n v\|_{L^\infty} 
\ee
		and for $x,x'\in \R^{N-1}$
\be \label{eq:Ao-est2}
		|\nu (x,r,\th )-\nu (x',r,\th ) |\leq \min (|x-x'|^{\gamma}, r^{\gamma}), \qquad
\ee
		and 
\be \label{eq:Ao-est3}
		|\d_o v(x,r,\th )- \d_o v(x',r,\th ) |\leq 2 \| \n v\|_{L^\infty }\min (|x-x'|,r ).
\ee
		Then,  by \eqref{eq:Ao-est1},	we have
		\begin{align}
		\| I_ov\|_{L ^{\infty}}&=\sup_{x\in\mathbb{R}^{N-1}} \int_{S^{N-2}}\int_0^\infty |\delta_o v(x,r,\theta)  \nu(x,r,\theta)|r^{-2-\a }drd\theta\nonumber\\
		&\leq C  \|\n v\|_{L^\infty} \int_0^\infty  \min (1, r^{\gamma}) r^{-1-\a} dr \leq C\|\n v\|_{L^\infty   }.\label{OperatorNorm001}
		\end{align}
We write
		\begin{align*}
			I_ov(x)-I_ov(z)
%
		%
		&=\int_{S^{N-2}}\int_{0}^\infty r^{-2-\alpha}\delta_o v(x,r,\theta)[\nu(x,r,\theta)-\nu(z,r,\theta)]\;drd\theta\\&+\int_{S^{N-2}}\int_{0}^\infty r^{-2-\alpha}[\delta_o v(x,r,\theta)-\delta_o v(z,r,\theta)]\nu(x,r,\theta)\;dr d\th.
		\end{align*}
From this, \eqref{eq:Ao-est1}, \eqref{eq:Ao-est2} and \eqref{eq:Ao-est3}, we get
		\begin{align*}
		|I_ov(x)-I_ov(z)|&\leq C \|\n v\|_{L^\infty } \int_{0}^\infty r^{-1-\alpha}\min\{r^{\gamma},|x-z|^{\gamma}\}\;dr \\
		&+C \|\n v\|_{L^\infty } \int_{0}^\infty r^{\gamma-2-\alpha}\min\{|x-z|,r\} \;dr \leq C \|\n v\|_{L^\infty  }|x-z|^{\gamma-\alpha}.
		\end{align*}
		Therefore
		\begin{align}
		[I_ov]_{C ^{\gamma-\alpha} }\leq C \|\n v\|_{L^\infty  }\label{OperatorNorm0002}.
		\end{align}
		The inequalities \eqref{OperatorNorm001} and \eqref{OperatorNorm0002} imply \eqref{eq:Io}.\\
For simplicity, we assume that 
$\|\mu\|_{ C^{\g}(\R^{N-1}\times [0,\infty)\times S^{N-2}) }\leq 1$.
We write 
\be \label{eq:dev-mean-v}
\d_e v(x,r,\th )=\frac{r}{2} \int_0^1[\n v(x-rt\th)- \n v(x+rt\th)]dt.
\ee
Hence, we have 
\be \label{eq:Ae-est1}
  |\d_e v(x,r,\th )  |\leq r \min (1, r^{\g}) \|\n v\|_{C^{\g}} 
\ee
		and for $x,x'\in \R^{N-1}$
\be \label{eq:Ae-est2}
		|\mu (x,r,\th )-\mu (x',r,\th ) |\leq \min (|x-x'|^{\gamma}, 1) 
\ee
		and 
\be \label{eq:Ae-est3}
		|\d_e v(x,r,\th )- \d_e v(x',r,\th ) |\leq  r \| \n v\|_{C^{\g}}\min (|x-x'|^\g,r^\g ).
\ee
We now estimate, using \eqref{eq:Ae-est1}, 
	\begin{align}
		\| I_ev\|_{L ^{\infty}}&=\sup_{x\in\mathbb{R}^{N-1}} \int_{S^{N-2}}\int_0^\infty |\delta_e v(x,r,\theta)   \mu (x,r,\theta)|r^{-2-\a }drd\theta\nonumber\\
		&\leq C  \|\n v\|_{C^{\g}} \int_0^\infty  \min (1, r^{\gamma}) r^{-1-\a}dr \leq C\|\n v\|_{C ^{\g} }.\label{OperatorNorm001-e}
		\end{align}
We have 
		\begin{align*}
			I_ev(x)-I_ev(z)&=\int_{S^{N-2}}\int_{0}^\infty r^{-2-\alpha}\delta_e v(x,r,\theta)[ \mu (x,r,\theta)- \mu (z,r,\theta)]\;drd\theta\\&+\int_{S^{N-2}}\int_{0}^\infty r^{-2-\alpha}[\delta_e v(x,r,\theta)-\delta_e v(z,r,\theta)] \mu (x,r,\theta)\;dr d\th.
		\end{align*}
From this, \eqref{eq:Ao-est1}, \eqref{eq:Ao-est2} and \eqref{eq:Ao-est3}, we get
		\begin{align*}
		|I_ev(x)-I_ev(z)|&\leq C \|\n v\|_{C^{\g} } \min\{1,|x-z|^{\gamma}\} \int_{0}^\infty r^{-1-\alpha}\ \min\{1,r^{\gamma}\};dr\\
		&+C \|\n v\|_{C^\g  } \int_{0}^\infty r^{-1-\alpha}\min\{|x-z|^\g,r^\g\} \;dr  \leq C \|\n v\|_{C^\g  }|x-z|^{\gamma-\alpha}.
		\end{align*}
This and \eqref{OperatorNorm001-e} give \eqref{eq:Ie}.
 The proof of $(i)$ is thus complete.\\
		
	For $(ii)$, 	we  assume that Supp$\n v \subset B_{R'}$, for some $R'>1$.
%
%
		We let $R>2R'$ and we  split
\begin{align*}
I_ev(x)&=\int_{S^{N-2}}\int_0^R \delta_e v(x,r,\theta)  \mu (x,r,\theta)r^{-2-\a }drd\theta+ \int_{S^{N-2}}\int_R^\infty \delta_e v(x,r,\theta)  \mu (x,r,\theta)r^{-2-\a }drd\theta\\
&=:I_e^1(x)+I_e^2(x).
\end{align*}
By  \eqref{eq:dev-mean-v}, 
\be\label{eq:Ie1-est-1}
 I_e^1(x)=0\qquad\textrm{ for $|x|\geq R'+R$.}
\ee
In addition, for $|x|>R$, by  \eqref{eq:dev-mean-v},
\begin{align}
|I_e^2(x)|\leq    \|\n v\|_{L^\infty } \int_0^1 \int_R^{\frac{|x|+R'}{t}}   r^{-1-\a }dr \, dt\leq C  \|\n v\|_{L^\infty}\left( R^{-\a}+( |x|+R')^{-\a}\right),
\end{align}
so that 
\be\label{eq:Ie1-est-2}
	\lim_{R\to\infty}   \|I_e\|_{L ^{\infty} } =0.
\ee
Now, for $h\in \R^{N-1}$, 
\begin{align*}
|I_e^2(x+h)- I_e^2(x)|&\leq   |h|^{\g-\a}  \|\n v\|_{C^{\g-\a}} \int_0^1 \int_R^{\frac{|x|+|h|+ R'}{t}}   r^{-1-\a }dr \, dt\\
&\leq C   |h|^{\g-\a}  \|\n v\|_{C^{\g-\a} }\left( R^{-\a}+ (|x|+|h|+R)^{-\a}\right).
\end{align*}
Combining this with  \eqref{eq:Ie1-est-2} and \eqref{eq:Ie1-est-1}, we deduce that  $\lim_{R\to\infty}   \|I_ev\|_{C ^{\gamma-\alpha}(\R^{N-1}\setminus B_R) }=0$. By similar argument we have  $\lim_{R\to\infty}   \|I_ov\|_{C ^{\gamma-\alpha}(\R^{N-1}\setminus B_R) }=0$, so that  \eqref{eq:tozero-} holds.\\

For  $(iii)$, we pick $v\in \cC ^{1+\g}_0 (\mathbb{R}^{N-1})$. Then there exists $v_n\in C^\infty_c(\R^{N-1})$ such that $v_n\to v$ in $C^{1+\g}(\R^{N-1})$.  Thanks to $(i)$, we have $I_e v_n,I_ov_n\in \cC ^{\g-\alpha}_0 (\mathbb{R}^{N-1})$. Now, by $(ii)$ and  Proposition \ref{prop:car-C_0}, we get  $I_ov_n , I_ev_n\in \cC ^{\b-\alpha}_0 (\mathbb{R}^{N-1})$. Once again by $(i)$,  we have $I_e v_n\to I_e v$ and $I_ov_n\to I_o v$ in $C^{\b-\a}(\R^{N-1})$. This then yields the conclusion, since  $\cC ^{\b-\alpha}_0 (\mathbb{R}^{N-1})$ is closed in $C^{\b-\a}(\R^{N-1})$.
\end{proof}		

        \subsection{A nonlocal comparison principle in the entire space}

In this subsection we provide a comparison principle for a class of nonlocal evolution equations relying on kernel assumptions in polar coordinates, which will be used later on.
        
\begin{proposition}\label{prop:mmp}
  Let $\a \in (0,1)$, $\b\in (\a,1)$,  $T>0$, $V\in C\left([0,T],C_b(\R^{N-1})\right)$ and let
   let $\mu \in C\left([0,T], C^\b( \R^{N-1}\times [0,\infty)\times S^{N-2})\right)$ with $\mu>0$ and  $\mu(t,x,0,\th)=\mu(t,x,0,-\th)$
Moreover, let $\beta >\alpha$, and let $u \in L^\infty([0,T], C^{\b}(\R^{N-1}))$ satisfy
$$
u(\cdot,x)\in C^1([0,T]) \quad \text{for every $x \in \R^{N-1}$,}\quad  
u(t,\cdot)\in C^{1+\b}_{loc}(\R^{N-1}) \quad \text{for every $t \in [0,T]$}
$$
and    
\be\label{eq:equ-sub-sol}
\de_t u(t,x)+P.V.\int_{\R^{N-1}}\frac{u(t,x)-u(t,x+y)}{|y|^{N+\a}} \mu(t, x,|y|,y/|y|) \, dy+ V(t,x)\cdot \n u(t,x)\leq 0 
\ee
 for $(t,x)\in [0,T]\times \R^{N-1}$.
Then $\displaystyle\sup_{[0,T]\times \R^{N-1}} u=\sup_{x\in \R^{N-1}} u(0,x).$
\end{proposition}
\begin{proof}
We consider $\eta\in C^\infty_c(B_2)$ with $0\leq \eta \leq 1$ on $\R^{N-1}$ and $\eta=1$ on $B_{1}$. We define, for $\e,R>0$, 
$$
v(t,x)=\eta_R(x) u(t,x)-\e t,
$$
where $\eta_R(x)=\eta(x/R)$.
We also define, for $w \in  C^{1+\b}_{loc}(\R^{N-1})\cap L^\infty(\R^{N-1})$ and $x \in \R^{N-1}$,
$$
\ov L(t) w(x)= P.V.\int_{\R^{N-1}}\frac{w(x)-w(x+y)}{|y|^{N+\a}} \mu(t,x,|y|,y/|y|) \, dy.
$$
We then  have 
\begin{align*}
&\ov L(t) (\eta_R u(t)) (x)
=\eta_R(x) \ov L(t)( u(t))(x) \\
&+ u(t,x) P.V.\int_{\R^{N-1}}\frac{ \eta_R(x) -\eta_R(x+y)}{|y|^{N+\a}} \mu(t, x,|y|, y/|y|) \, dy\\
&+ P.V.\int_{\R^{N-1}}\frac{ \eta_R(x) -\eta_R(x+y)}{|y|^{N+\a}} (u(t,x+y)-u(t,x))\mu(t, x,|y|, y/|y|) \, dy.
\end{align*}
From this, we get 
\begin{align*}
&\ov L(t) (\eta_R u(t)) (x)
- \eta_R(x) P.V.\int_{\R^{N-1}}\frac{ u(t,x)-u(t, x+y)}{|y|^{N+\a}} \mu(t,x,|y|, y/|y|) \, dy\\
&= \frac{u(t,x)}{2} P.V.\int_{\R^{N-1}}\frac{ 2\eta_R(x)-\eta_R(x-y) -\eta_R(x+y)}{|y|^{N+\a}} \mu(t, x,0, y/|y|) \, dy\\
&+{u(t,x)} P.V.\int_{\R^{N-1}}\frac{ \eta_R(x) -\eta_R(x+y)}{|y|^{N+\a}}[ \mu(t, x,|y|, y/|y|)-\mu(t, x,0, y/|y|)] \, dy\\
&+ P.V.\int_{\R^{N-1}}\frac{ \eta_R(x) -\eta_R(x+y)}{|y|^{N+\a}} (u(t,x+y)-u(t,x))\mu(t, x,|y|, y/|y|) \, dy.
\end{align*}
Hence,  by Lemma \ref{SemigroupLemma001} we  get 
$$
 \|\ov L(t) (\eta_R u(t))- \eta_R\ov L(t)  u(t)\|_{L^\infty(\R^{N-1})}\leq C\|\n \eta_R\|_{C^1_b} \|u(t)\|_{C^{\b}}\leq \frac{C}{R} \|u(t)\|_{C^{\b}}
$$
for $t \in [0,T]$. From this and \eqref{eq:equ-sub-sol}, we then obtain 
\be\label{eq:eq-fro-v-cut}
\de_t v+\ov L(t) v+ V\cdot \n v\leq -\e + F_R \qquad\textrm{ in $[0,T]\times \R^{N-1}$,}
\ee
with
\be\label{eq:estimF_R-mmp}
\|F_R\|_{L^\infty((0,T)\times \R^{N-1})}\leq \frac{C  \|u\|_{L^\infty((0,T),C^{\b})}}{R}(1+\|V\|_{L^\infty((0,T)\times \R^{N-1})}).
\ee
We claim that 
\be \label{eq:mmpv-v0}
\max_{[0,T]\times \R^{N-1}} v=\max_{ x\in  \R^{N-1}} v(0,x).
\ee
Indeed, let  $(t_0,x_0)\in [0,T]\times \R^{N-1}$ be such that $v(t_0,x_0)=\max_{[0,T]\times \R^{N-1}} v$. Suppose that $t_0>0$.  The maximality property then implies that $ \n_x v(t_0,x_0) =0$, $\de_t v(t_0,x_0)\geq 0$ and also $L(t_0)v(x_0) \ge 0$, since $\mu \ge 0$ by assumption. By \eqref{eq:eq-fro-v-cut} we thus have 
$$
0\leq \ov L(t_0)v(x_0) \leq  -\e+\|F_R\|_{L^\infty((0,T)\times \R^{N-1})}
$$
which is not possible if $R$ is large enough, thanks to \eqref{eq:estimF_R-mmp}. Therefore $t_0=0$ and  thus \eqref{eq:mmpv-v0} holds   as claimed. \\
  Letting now $R\to \infty$ and then  $\e\to 0$ in \eqref{eq:mmpv-v0}, we get the result.
\end{proof}

	\section{Analytic semigroups, their generators, intermediate spaces and associated semilinear evolution equations}\label{Section3}

	We begin by introducing some notions regarding function spaces and analytic semigroups (see e.g. \cite[Chapter 2]{lunardi2012analytic}). 
        For normed vector spaces $\cX,\cX'$ we let $\cL(\cX,\cX')$ denote the space of continuous linear operators $\cX \to \cX'$, endowed with the usual operator norm. As usual, we also set $\cL(\cX):= \cL(\cX,\cX)$.

A strongly continuous analytic semigroup on a Banach space $\cX$ is a family of operators $\{T(t)\}_{t \ge 0} \subset \cL(\cX)$ with the following properties: 
\begin{itemize}
\item[(i)] $T(0)=\id$, $T(t+s)= T(t)T(s)$ for all $t,s \ge 0$.
\item[(ii)] The function $(0,\infty) \to \cL(\cX)$, $t \mapsto T(t)$ is analytic.
\item[(iii)] The function $[0,\infty) \to \cX$, $t \mapsto T(t)u$ is continuous for every $u \in \cX$.
\end{itemize}

The generator of such a semigroup is given as an (unbounded) linear operator $B_0: \cY \subset \cX \to \cX$ by
$$
B_0 u := \lim_{t \to 0^+} \frac{T(t)u-u}{t}, \qquad u \in  \cY,
$$
where the domain $\cY \subset \cX$ of $B_0$ is given as the subspace of $u \in \cX$ for which this limit exists.

We recall that an operator $B_0: \cY \subset \cX \to \cX$ generates a strongly continuous analytic semigroup in this sense if and only if $B_0$ is sectorial (see e.g. \cite[Def. 2.0.1]{lunardi2012analytic} for a definition) and its domain $\cY$ is dense in $\cX$. In such a case, the operator $B_0$ is also closed, which means that $\cY$ is a Banach space with the graph norm 
$u \mapsto \|u\|_{\cX} + \|\cB_0 u\|_{\cX}$ on $\cY$. As a consequence, by the open mapping theorem, the graph norm on $\cY$ is equivalent to any other given norm $\|\cdot\|_{\cY}$ on $\cY$ for which $(\cY, \|\cdot\|_{\cY})$ is a Banach space. 

In the following, if $\cB_0: \cY \subset \cX \to \cX$ generates a strongly continuous analytic semigroup, we shall denote this semigroup by
$$
t \mapsto e^{B_0 t} \in \cL(\cX),\qquad t \ge 0.
$$
As noted in \cite[Proposition 2.1.1]{lunardi2012analytic},   for all  $k\in \mathbb{N}$,
	there exists $M_k>0$, such that
	\begin{align}\label{eq:B_0k-M_k}
	\|t^kB_0^ke^{B_0t}\|_{\mathcal{L}(Y)}\leq M_k,\hspace{.5cm}\mbox{for all  $t\in (0,1]$.}
	\end{align}
In order to obtain optimal regularity estimates in time, it is convenient to  introduce, for $\rho\in [0,1)$, the intermediate space
	\begin{align}\label{eq:def-interm-spcace}
	\mathcal{D}_{B_0}(\rho,\infty)=\{f\in \cY:[f]_{\mathcal{D}_{B_0}(\rho,\infty)}=\sup_{0<t\leq 1}\|t^{1-\rho}B_0e^{B_0t}f\|_\cY<\infty\},
	\end{align}
        endowed with the norm $\|f\|_{\calD_{B_0}(\rho,\infty)}=\|f\|_\cY+[f]_{\mathcal{D}_{B_0}(\rho,\infty)}.$
           
         We then have the following result taken from \cite[Theorem 4.3.1 (iii)]{lunardi2012analytic}.
	\begin{theorem} \label{thhhh:Corollary0002}
Let $B_0: \cY \subset \cX \to \cX$ be the generator of a strongly continuous analytic semigroup. Let   $T>0$, $\rho\in (0,1)$, $u_0\in \cX$, and let 
		$f\in C^\rho([0,T],\cY)$   be such that  $f(0)-B_0u_0 \in {\mathcal{D}_{B_0}(\rho,\infty)}  $.   Then the problem 
		\begin{equation*}
		\left\{
		\begin{array}{rll}
		u'(t)+B_0u(t)&=&f(t), \hspace{1cm}  t\in(0,T]\\
		u(0)&=&u_0
		\end{array}
		\right.
		\end{equation*}
admits a unique   solution $u\in C^\rho([0,T],\cX)\cap C^{1+\rho}([0,T],\cY)$. Moreover,
 there exists  $C_T=C(T,\rho,M_0,M_1,M_2)>0$ such that
			\begin{align}\label{eq:fdgeB_0}
		\| B_0 (u- u_0)\|_{C^\rho([0,T],\cY) }+ &	\| u-u_0\|_{  C^{1+\rho}([0,T],\cY)} \nonumber\\
		&\leq C_T \Big(\|f- B_0 u_0\|_{C^\rho([0,T],\cY)}  +\|f(0)-B_0u_0 \|_{\mathcal{D}_{B_0}(\rho,\infty)}\Big)
\end{align} 
and  $C_T\leq C_{T_0}$ for all $T\leq T_0$.
	\end{theorem}
	\begin{proof}
	The existence,  uniqueness and
	 \eqref{eq:fdgeB_0} follow from  \cite[Theorem 4.3.1]{lunardi2012analytic}. Moreover,  as explained in the beginning of Section 4.1 in \cite{lunardi2012analytic},  the constant $C_T$ is  increasing in $T$. 	
	\end{proof}

        \begin{remark}
          \label{intermediate-spaces-remark}
          Let $B_0: \cY \subset \cX \to \cX$ be the generator of a strongly continuous analytic semigroup and $\s \in (0,1)$. Then, as noted in \cite[Proposition 2.2.2]{lunardi2012analytic}, the intermediate space $\mathcal{D}_{B_0}(\rho,\infty)$ does not depend on the operator $B_0$ itself, as it coincides with a real interpolation space between the spaces $\cY$ and $\cX$ (with equivalence of respective norms). We shall use this fact in the following where we consider a H\"older space setting.
          
        \end{remark}

        \subsection{Intermediate spaces in a H\"older space setting}\label{sec:interm-spac-hold}

        In our application to the nonlocal mean curvature flow, we will need to consider the special case where $\cY= \cC^{1+\b}_0 (\mathbb{R}^{N-1})$ and $\cX = \cC^{\b-\s}_0(\mathbb{R}^{N-1})$ for values $\s\in (0,1)$,  $\b\in (\s,1)$. Our next result provides a characterization of the  intermediate space $\mathcal{D}_{B_0}(\rho,\infty)$ defined in \eqref{eq:def-interm-spcace} in this particular case. As noted in Remark~\ref{intermediate-spaces-remark}, this space does not depend on the particular choice of a generator $B_0: \cY \subset \cX \to \cX$ of a strongly continuous analytic semigroup. 

\begin{proposition}\label{prop:car-inerpol-space}
Let $\s\in (-1,1)$,  $\b\in (\s,1+\s)$ with $\b\not\in \N$. Let    $B_0:\cC^{1+\b}_0 (\mathbb{R}^{N-1}) \subset  \cC^{\b-\s}_0(\mathbb{R}^{N-1}) \to \cC^{\b-\s}_0(\mathbb{R}^{N-1})$ be any infinitesimal generator of a strongly continuous analytic semigroup on $ \cC^{\b-\s}_0(\mathbb{R}^{N-1})$. Let  $\rho\in (0,\min(1,\frac{1}{1+\s}))$ and put $ \g_\rho:=\b+\rho(1+\s).$   
 Then $ \cC^{\b-\s}_0(\mathbb{R}^{N-1})\cap C^{\g_{\rho}-\s}({\R^{N-1}})\subset  \mathcal{D}_{B_0}(\rho,\infty) $ and  there exists $c=c(\s,\b,\rho)>0$ such that
\be \label{eq:int-incl}
   \|f\|_{\mathcal{D}_{B_0}(\rho,\infty) }\leq c \|f\|_{C^{\g_\rho-\s}} \qquad\textrm{for all $f\in  \cC^{\b-\s}_0(\mathbb{R}^{N-1})\cap C^{\g_{\rho}-\s}({\R^{N-1}})$.}
\ee
If moreover $\g_\rho-\s \not\in \N$, then   $ \cC^{\b-\s}_0(\mathbb{R}^{N-1})\cap C^{\g_{\rho}-\s}({\R^{N-1}})= \mathcal{D}_{B_0}(\rho,\infty) $, with equivalence of their respective norms.  
\end{proposition}
\begin{proof}
By \cite[Corollary 2.17]{abels2009cauchy}, for $\s\in (-1,1)$,  the standard fractional Laplacian $  (-\D)^{\frac{1+\s}{2}}: \cC^{1+\b}_0 (\mathbb{R}^{N-1})\to \cC^{\b-\s}_0(\mathbb{R}^{N-1})$  is a  generator of a strongly continuous analytic semigroup, and therefore 
\be\label{eq:DB0-DmbL}
 \mathcal{D}_{B_0}(\rho,\infty)  = \mathcal{D}_{(-\D)^{\frac{1+\s}{2}}}(\rho,\infty)  ,
\ee
with equivalence of their respective norms,  see Remark~\ref{intermediate-spaces-remark}, .

For simplicity, we write $\g=\g_\rho$.
Letting $f\in  \cC^{\b-\s}_0(\mathbb{R}^{N-1})\cap C^{\g-\s}({\R^{N-1}})$, we then  have 
\be \label{eq:uxt}
e^{t (-\D)^{\frac{1+\s}{2}}}f(x)=  K(\cdot, t)\star f (x)=\int_{\R^{N-1}} K(x-y, t) f(y) dy,
\ee
where $K$ is the heat kernel of $(-\D)^{\frac{1+\s}{2}}$. 
It is known, see e.g. \cite{BGet},  that
$$
K(x,t)=t^{-\frac{N-1}{1+\s}}P(t^{-\frac{1}{1+\s}}x),  
$$
for some    radially  symmetric function  $P\in C^\infty(\R^{N-1})$, with
$$
  \qquad |D^k P(y)| \leq \frac{C(k,N,\s)}{1+|y|^{N+\s+k}}.
$$
From this,  we can apply  \cite[Lemma 2.1]{Fall-Weth} to get
\be \label{eq:Fall-Weth}
|(-\D)^{\frac{1+\s}{2}} P (z)|\leq  \frac{C}{1+|z|^{N+\s}}.
\ee
We have, using that $\int_{\R^{N-1}}(-\D)^{\frac{1+\s}{2}} P(z)\, dz=0$ and a change of variables, 
\begin{align*}
H_t(x)&:=t (-\D)^{\frac{1+\s}{2}} (e^{t (-\D)^{\frac{1+\s}{2}}}f)(x)=t \int_{\R^{N-1}}  (-\D)^{\frac{1+\s}{2}} K(x-y, t) f(y)\,  dy\\
&= t^{-\frac{N-1}{1+\s}}  \int_{\R^{N-1}} ( (-\D)^{\frac{1+\s}{2}} P) (t^{-\frac{1}{1+\s}}(x-y)) f(y)\, dy\\
&= \int_{\R^{N-1}} (-\D)^{\frac{1+\s}{2}} P (z) f(x-t^{\frac{1}{1+\s}} z)\, dz=  \int_{\R^{N-1}}  (-\D)^{\frac{1+\s}{2}} P (z) [f(x-t^{\frac{1}{1+\s}} z)-f(x)]\, dz.
\end{align*}
Now using that $ (-\D)^{\frac{1+\s}{2}} P  $ is even, we conclude that  
\be\label{eq:H-tx-even}
H_t(x)=  \frac{1}{2} \int_{\R^{N-1}} (-\D)^{\frac{1+\s}{2}} P (z) [f(x-t^{\frac{1}{1+\s}} z)+f(x+t^{\frac{1}{1+\s}} z)-2f(x)]\, dz.
\ee
Suppose that $\g-\b \leq 1$.
We then deduce,  from  that 
\begin{align}\label{eq:L-infty-boundHt}
|H_t(x) |\leq   t^{\frac{\g-\b}{1+\s}}  C [f]_{C^{\g-\b}}\int_{\R^{N-1}}  |z|^{\g-\b} | (-\D)^{\frac{1+\s}{2}} P (z)| \, dz .
\end{align}
 If also $\g-\b>1$ then since 
\begin{align*}
|[f(x-t^{\frac{1}{1+\s}} z)&+f(x+t^{\frac{1}{1+\s}} z)-2f(x)] |\leq 2  [\n f]_{C^{\g-\b-1} } t^{\frac{1}{1+\s}}| z |  \min (1, (t^{\frac{1}{1+\s}}  |z|)^{\g-\b-1}),
\end{align*}
we  still have \eqref{eq:L-infty-boundHt}. As a consequence 
\begin{align}\label{eq:L-infty-boundHt-ok}
  \|H_t\|_{L^\infty}\leq   C t^{\rho} \|f\|_{C^{\g-\s}} \qquad\textrm{for  all $\rho\in (0,\min(1,\frac{1}{1+\s})).$}
\end{align}
To proceed, we start with the case $\g-\s\leq 1$. Then
 for $x\not=x'\in\R^{N-1}$, by \eqref{eq:Fall-Weth} and the fact that $0<\g-\b=\rho(1+\s)<1+\s$,   we have 
\begin{align*}
&|H_t(x)- H_t(x')|\leq  C [f]_{C^{\g-\s}} \int_{\R^{N-1}} | (-\D)^{\frac{1+\s}{2}} P (z)| \min (|x-x'|, t^{\frac{1}{1+\s}}  |z|)^{\g-\s} \, dz\\
&\leq   C [f]_{C^{\g-\s}}   t^{\frac{\g-\s}{1+\s}}   \int_{ |z|\leq  t^{-\frac{1}{1+\s}}  |x-x'|} |z|^{\g-\s} |  (-\D)^{\frac{1+\s}{2}}P (z)|    \, dz\\\
&+   C [f]_{C^{\g-\s}}  |x-x'|^{\g-\s}  \int_{  |z|\geq   t^{\frac{-1}{1+\s}} |x-x'|} |(-\D)^{\frac{1+\s}{2}} P (z)|  \, dz\\
&\leq    C [f]_{C^{\g-\s}}  t^{\frac{\g-\s}{1+\s}}  t^{-\frac{\b-\s}{1+\s}}  |x-x'|^{\b-\s}  \int_{\R^{N-1} } |z|^{\g-\b} | (-\D)^{\frac{1+\s}{2}}P (z)|    \, dz\\
&+ C [f]_{C^{\g-\s}}   |x-x'|^{\g-\s}  \int_{  |z|\geq   t^{\frac{-1}{1+\s}} |x-x'|}\frac{1}{1+  |z|^{N-1+(\g-\b)}} \, dz\\
&\leq   C [f]_{C^{\g-\s}} \left(  t^{\frac{\g-\b}{1+\s}}  |x-x'|^{\b-\s}  +   |x-x'|^{\g-\s} ( t^{\frac{-1}{1+\a}} |x-x'|)^{-(\g-\b)} \right)\\
&\leq C [f]_{C^{\g-\s}}     t^{\frac{\g-\b}{1+\s}}  |x-x'|^{\b-\s}.
\end{align*}
This clearly  implies  that  
$$
\sup_{t\in (0,1)}  t^{-\rho} [H_t]_{C^{\b-\s} }\leq C [f]_{C^{\g-\s}(\mathbb{R}^{N-1})}.
$$ 
As a consequence, using also \eqref{eq:L-infty-boundHt-ok}, we have  
\be \label{eq:L-infty-bde-s1} 
\sup_{t\in (0,1)}  t^{-\rho} \|H_t\|_{C^{\b-\s} }\leq C  \| f\|_{C^{\g-\s} } \qquad\textrm{ for $\g-\s\leq 1.$}
\ee
We now assume that $   \g-\s>1$ and we observe  that from the upper bounds of $\rho$ and $\b$ we have $\g-\s\leq 2$. 
Therefore (recalling \eqref{eq:H-tx-even})    using that
\begin{align*}
|[f(x-t^{\frac{1}{1+\s}} z)&+f(x+t^{\frac{1}{1+\s}} z)-2f(x)]-[f(x'-t^{\frac{1}{1+\s}} z)+f(x'+t^{\frac{1}{1+\s}} z)-2f(x')]|\\
&\leq 2  [\n f]_{C^{\g-\s-1} } t^{\frac{1}{1+\s}}| z |  \min (|x-x'|, (t^{\frac{1}{1+\s}}  |z|))^{\g-\s-1}
\end{align*}
   and the same argument as above,  we  get 
$$ 
\sup_{t\in (0,1)}  t^{-\rho} [H_t]_{C^{\b-\s} }\leq C  [\n f]_{C^{\g-\s-1} }.
$$
From this,   \eqref{eq:L-infty-boundHt-ok}  and \eqref{eq:L-infty-bde-s1},  we conclude    that for all $\rho\in (0,\min(1,\frac{1}{1+\s}))$,
$$
\|f\|_{\mathcal{D}_{(-\D)^{\frac{1+\s}{2}}}(\rho,\infty) }=\|f\|_{C^{\b-\s}}+\sup_{t\in (0,1)}  t^{-\rho} \|H_t\|_{C^{\b-\s} }\leq C \|f\|_{C^{\g-\s} }.
$$
Thanks to \eqref{eq:DB0-DmbL} and \eqref{eq:L-infty-boundHt} we obtain
$$
\|f\|_{\mathcal{D}_{B_0}({\rho},\infty) } \leq C \|f\|_{C^{\g-\s} }.
$$
Therefore $ \cC^{\b-\s}_0(\mathbb{R}^{N-1})\cap C^{\g_{\rho}-\s}({\R^{N-1}})\subset \mathcal{D}_{B_0}(\rho,\infty) $ and     \eqref{eq:int-incl} holds. \\
Next,  by  \cite[Proposition 2.2.2]{lunardi2012analytic} and \cite[Corollary 1.2.18]{lunardi2012analytic} we have that  $\mathcal{D}_{B_0}(\rho,\infty)$ is continuously embedded in $C^{\g_\rho-\s}(\mathbb{R}^{N-1}) $, provided $\g_\rho-\s\not\in\N$. Since, by definition $\mathcal{D}_{B_0}(\rho,\infty)\subset \cC^{\b-\s}_0(\mathbb{R}^{N-1})$, we get the desired result.

\end{proof}
\begin{remark}
We point out that      Proposition \ref{prop:car-inerpol-space} still holds when $\b\in \{0,1\}$ and   $\g_\rho-\s=1$, provided that, for $k\in \{1,2\}$, the spaces  $C^{k}(\mathbb{R}^{N-1})$ and $\cC^{k}_0(\mathbb{R}^{N-1})$ are, respectively, replaced with the H\"older-Zygmund space $\cC^{k}(\mathbb{R}^{N-1})$ and the space $\cC^{k}_0(\mathbb{R}^{N-1})=\ov{ C^\infty_c(\R^{N-1}) }^{\|\cdot\|_{\cC^{k}}}$. Recall that the  space $\cC^{1}(\mathbb{R}^{N-1})$ is defined by 
\begin{align}
\cC^1 (\mathbb{R}^{N-1}):=\bigg\{u\in C_b(\mathbb{R}^{N-1}):\,[u]_{\cC^{1 } }:=\sup_{\stackrel{x,y\in\mathbb{R}^{N-1}}{x\neq y}}\frac{|u(x)-2u(\frac{x+y}{2})+u(y)|}{|x-y|}<\infty\bigg\}
\end{align}
and  $\cC^2 (\mathbb{R}^{N-1})$ is given by the set of $u\in \cC^1 (\mathbb{R}^{N-1})$ such that $\de_i u\in \cC^1 (\mathbb{R}^{N-1}) $ for $i=1,\dots, N-1$.
\end{remark}

\subsection{A class of nonlocal operators generating strongly continous analytic semigroups}\label{ss:general-nonloc}
	In this section, we consider a class of nonlocal operators which  we prove to generate strongly continuous analytic semigroups.\\
	 For fixed $\g\in (\a,1)$,  we consider linear nonlocal operator $	\mbL_K: C^{1+\g}(\R^{N-1})\to C^{\g-\a}(\R^{N-1})$ given by 
         \be\label{eq:class-op}
         \mbL_Ku(x)=P.V.\int_{\R^{N-1}}\frac{u(x)-u(y)}{|x-y|^{N+\a}}K(x,y)\, dy,
\ee
	where $K:\big(\mathbb{R}^{N-1}\times\mathbb{R}^{N-1}\big)\setminus\{(x,x):x\in\mathbb{R}^{N-1}\}\rightarrow\mathbb{R}$ is a measurable function satisfying the following assumptions.
\begin{assumptions}\label{ass:L}$ $
{\rm
	\begin{enumerate}
	\item[(i)] $K(x,y)=K(y,x)$ for all $x,y\in \R^{N-1}$, $x\neq y$.
	\item[(ii)] $\frac{1}{\k}\leq K(x,y)\leq \k$,  for some $\k>1$.
	\item[(iii)] there exists $A_K\in C^\g(\R^{N-1}\times [0,\infty)\times S^{N-2})$, with $A_K(x,0,\th)=A_K(x,0,-\th)$ such that  $  A_{K}(x,r,\th)=K(x,x-r\th)$ for all $r>0$ and $\|A_K\|_{C^\g(\R^{N-1}\times [0,\infty)\times S^{N-2})}\leq \k$.
	\end{enumerate}
	}
	\end{assumptions}
	In order to apply the estimates in Lemma~\ref{SemigroupLemma001}, it will be useful to decompose $\mbL_K$ in two parts, writing   
	\be\label{eq:op-mbL-polar}
 \mbL_Ku(x)=\int_{S^{N-2}}\int_0^\infty\frac{\d_eu(x,r,\th)}{r^{2+\alpha}} {A}_K^e(x,r,\theta)drd\theta+	 \int_{S^{N-2}}\int_0^\infty\frac{\d_o u(x,r,\th)}{r^{2+\alpha}} {A}_K^o(x,r,\theta)drd\theta,
	\ee
for $u \in C^{1+\g}(\R^{N-1})$ with 
		$$
		 {A}_K^e(x,r,\theta)	 = A_K(x,r,\theta)+A_K(x,r,-\theta),   \qquad
		 {A}_K^o(x,r,\theta)	 = A_K(x,r,\theta)-A_K(x,r,-\theta)
		$$
                and $\delta_e u$, $\delta_o u$ defined as in \eqref{eq:def-deeo}. We now define the set 
\be
\calO_{\nu}^\g:=\{ u\in C^{1+\gamma}_{loc}(\mathbb{R}^{N-1})\,:\, \|\n u\|_{C^\g}\leq \nu \},
\ee
for $\nu>0$, and we state the following estimates.
		\begin{lemma}\label{lem:mbL-graph-norm}
			\begin{enumerate}
				\item Let $\g>\a$. Then there exist $C,C'>0$ depending only on $N,\alpha,\gamma,\kappa$  such that
				\be \label{eq:norm-nu-mbl}
				\| \mbL_K u\|_{ C^{\gamma-\alpha} }\leq C'\|\n u\|_{C^\g}  \qquad \text{for all $u\in \calO_{\nu}^\g$}      
				\ee
				and
				\be  \label{eq:semi-norm-nu-mblllll}
				\|\n u\|_{ C^{\g}  }\leq C(\| \mbL_K u\|_{ C^{\g-\alpha} }+ \|\n u\|_{ L^{\infty}} ) \qquad \text{for all $u\in \calO_{\nu}^\g$.}
				\ee
			\item	If  $\b\in (\a,\g)$, then  we have $\mbL_K u \in \cC^{\b-\alpha}_0 (\mathbb{R}^{N-1})$ for all $ u\in \cC^{1+\b}_0 (\mathbb{R}^{N-1})\cap\calO_\nu^\g$.
			\end{enumerate}
	\end{lemma}
		
	\begin{proof}
(i) It is clear from Lemma \ref{SemigroupLemma001} and \eqref{eq:op-mbL-polar} that  \eqref{eq:norm-nu-mbl} holds.\\
			Now  we  let $f(x)=\mbL_K u(x)$.  Since $K(\cdot+x_0, \cdot+x_0)$ satisfies also Assumptions \ref{ass:L} for all $x_0\in \R^{N-1}$,  by  \cite[Theorem 1.3$(ii)$, Theorem 1.4$(iii)$]{fall2020regularity}, we get 
			\begin{align*}
		\|\nabla u\|_{C^\g(B_1(x_0))}&=	\|\n (u-u(x_0))\|_{C^\g(B_1(x_0))}\\&\leq C\left(\| f\|_{C^{\g-\a}(B_2(x_0))}+ \| u-u(x_0)\|_{L^\infty(B_2(x_0))} +\int_{\R^{N-1}}\frac{|u(x)-u(x_0)|}{1+|x-x_0|^{N+\a}}\, dy\right) \\
			&\leq  C\left(\| f\|_{ C^{\gamma-\alpha} }+ \| \n u\|_{L^\infty}+\int_{|x-x_0|\geq 2}\frac{|u(x)-u(x_0)|}{|x-x_0|^{N+\a}}\, dy\right)\\&\leq  C\left(\| f\|_{ C^{\gamma-\alpha} }+ \| \n u\|_{L^\infty}+\| \n u\|_{L^\infty}\int_{|x-x_0|\geq 2}\frac{1}{|x-x_0|^{N-1+\a}}\, dy\right).
			\end{align*}
			From this, we then get 
			$$
			\|\n u\|_{C^\g(B_1(x_0))}\leq C\left(\| f\|_{ C^{\gamma-\alpha} }+  \|\n u\|_{L^\infty}  \right),
			$$
			where $C$ may change value  from one line to an other.
			Since $x_0$ is arbitrary,  \eqref{eq:semi-norm-nu-mblllll} follows.\\
                        (ii) By  Lemma \ref{SemigroupLemma001}  and \eqref{eq:op-mbL-polar} if  $ u\in \cC^{1+\b}_0(\mathbb{R}^{N-1}) $, then $\mbL_K u \in \cC^{\b-\alpha}_0 (\mathbb{R}^{N-1})$.
	\end{proof}
%

Combining Lemma~\ref{lem:mbL-graph-norm} and a result from \cite{abels2009cauchy}, we now exhibit a useful class of operators generating strongly continuous analytic semigroups in H\"older spaces.
	\begin{proposition}\label{lem:mbbL_analytic}
Let   $\g\in (\b,1)$ and $b\in {\cC^{\b-\alpha}_0}(\mathbb{R}^{N-1})$ be such that 
$$
\frac{1}{\k}\leq |b(x)|\leq  \k  \qquad\textrm{ for all $x\in \R^{N-1}$}.
$$  
Let $L:{\cC^{1+\b}_0}(\mathbb{R}^{N-1})\to {\cC^{\b-\alpha}_0}(\mathbb{R}^{N-1})$ be a bounded linear operator satisfying for all $\e>0$, there exists $c_\e$ such that 
	\begin{align}\label{eq:L-perturb-mbL}
		\|Lu\|_{C^{\b-\alpha}} \leq \varepsilon\|\mbL_K u\|_{C^{\b-\alpha}}+c_\e\|u\|_{C^{\b-\alpha}}\;\;\;\mbox{ for all }u\in \cC^{1+\b}_0(\mathbb{R}^{N-1}).
		\end{align}
Suppose that for all $k\in\N $,
\be \label{eq:smooth-even}
 \sup_{\th\in S^{N-2}, }\|\n^k_\th A_K^e(\cdot,0,\th)\|_{C^\g}\leq C_k.
\ee	
Then, the operator  	
\be \label{eq:bmbLKL}
 b(x)\mbL_K+L:\cC^{1+\b}_0(\mathbb{R}^{N-1})\to \cC^{\b-\alpha}_0(\mathbb{R}^{N-1}),
 \ee
  with domain ${\cC^{1+\b}_0} (\mathbb{R}^{N-1})$,   generates a strongly continuous   analytic semigroup on ${\cC^{\b-\alpha}_0}(\mathbb{R}^{N-1})$.
	\end{proposition}
\begin{proof}
Recalling \eqref{eq:op-mbL-polar}, we write  
\be 
b(x)\mbL_K:= L_1+ L_2,
\ee
where 
\begin{align*}
L_1u(x)&=b(x) \Bigl(P.V.\displaystyle\int_{\mathbb{R}^{N-1}}\frac{(u(x)-u(x+y))}{|y|^{N+\alpha}}  A_K^e(x,0,y/|y|)\;dy\Bigr)\\
&=\frac{b(x)}{2} \int_{\mathbb{R}^{N-1}}\frac{(2u(x)-u(x+y)-u(x-y))}{|y|^{N+\alpha}}  A_K^e(x,0,y/|y|)\;dy
\end{align*}
and 
$$L_2u(x)=\displaystyle\int_{\mathbb{R}^{N-1}}\frac{(u(x)-u(x+y))}{|y|^{N+\alpha}} b(x)[A_K^o(x,|y|,y/|y|)+ A_K^e(x,|y|,y/|y|)-A_K^e(x,0,y/|y|)]\;dy.$$ 
By  Lemma \ref{SemigroupLemma001},   the operator   $L_1:{\cC^{1+\b}_0}(\mathbb{R}^{N-1})\to {\cC^{\b-\alpha}_0}(\mathbb{R}^{N-1})$ is bounded.\\

%
%
%
Letting  $k(x,y):=b(x)A_K^e(x,0,y/|y|)|y|^{-N-\alpha}$ and $\chi \in C^\infty_c(\R^{N-1})$ be nonnegative,  radially symmetric and  satisfy $\chi=1$ for $|y|\leq 2$. One can split $k $ as $k =k_1+k_2$, where $k_ 1(x,y)=k (x,y)\chi(y)$ and $k_2(x,y)=k (x,y)(1-\chi(y))$. Then,  it is easy to see that $k_ 1=0$ for $|y|\geq 2$ and  $k_2=0$ for $|y|\leq 1$. In addition, using also \eqref{eq:smooth-even}, for all $n\in\mathbb{N}^{N-1}$ with  $|n|\leq N$,   there exists $C>0,c>0$ only  depending on $N,\alpha,\g,n$ and $\k$ such that
	\begin{equation}
		\left\{
		\begin{aligned}
		&\|\partial_y^nk_ 1(\cdot,y)\|_{C^{\g} }\leq C|y|^{-N-\a-|n|},\qquad 0<|y|\leq 2,\\
	&  k_ 1(x,y)\leq C |y|^{-N-\alpha}, \qquad \; 0<|y|\leq 1,\; x\in\mathbb{R}^{N-1},\\
&\|k_ 2(\cdot,y)\|_{C^{\g} }\leq C|y|^{-N-\alpha'},\qquad 0<|y|\leq 1,\;\;\mbox{ for }\;\;0\leq \alpha'<\alpha<1\\
	&\displaystyle	\int_{|y|\geq 1}\|k_ 2(\cdot,y)\|_{C^{\g} }\;dy <\infty,\\
 		&\lim\limits_{|y|\rightarrow\infty}\|k_ 2(\cdot,y)\|_{C^{\g} }=0.
		\end{aligned}
		\right.
	\end{equation}
In view of this, we can apply  \cite[Corollary 2.17]{abels2009cauchy}, to deduce that, for  $\b\in (\a,\g)$, the operator   $L_1:{\cC^{1+\b}_0}(\mathbb{R}^{N-1})\to {\cC^{\b-\alpha}_0}(\mathbb{R}^{N-1})$ generates a strongly continuous analytic semigroup.\\
By  Lemma \ref{SemigroupLemma001} and  \eqref{eq:interpol}, for all $\varepsilon>0$, there exists  $C_2(\varepsilon)>0$ such that
			\begin{align*}
			\|L_2u\|_{C^{\b-\alpha} }\leq C \|\n u\|_{C_{b } } \leq  C\left(\varepsilon \|u\|_{C^{1+\b} }+  C_2(\varepsilon) \|u\|_{L^{\infty} }\right).
			\end{align*}
			Combining this information with  Lemma \ref{lem:mbL-graph-norm} (i), we get 
					\begin{align*}
			\|L_2u\|_{C^{\b-\alpha} } \leq  C\left(\varepsilon \|(L_1+L_2)u\|_{C^{\b-\alpha} }+   C_2(\varepsilon)\|u\|_{L^{\infty}  }\right).
			\end{align*}
	 with a possibly different constant $C$ not depending on $\varepsilon$.
	 The above two estimates give
				\begin{align*}
			\|L_2u\|_{C^{\b-\alpha} } \leq  C\left(\varepsilon \|L_1u\|_{C^{\b-\alpha} }+  C_2(\varepsilon) \|u\|_{L^{\infty}  }\right).
			\end{align*}
		From this and \eqref{eq:L-perturb-mbL}, we have 
			\begin{align*}
		\|(L_2+L)u\|_{ C^{\b-\alpha} }\leq \varepsilon\|L_1u\|_{ C^{\b-\alpha} }+C(\varepsilon)\|u\|_{C^{\b-\alpha} }\;\;\;\mbox{ for all }u\in {\cC^{1+\b}_0(\mathbb{R}^{N-1})}.
		\end{align*}		
Now applying \cite[ Theorem 2.1]{pazy2012semigroups}, we deduce that  $b (x)\mbL_K+L=L_1+L_2+L$ is a  generator of a strongly continuous analytic semigroup. 	
\end{proof}

We shall now derive the following existence and uniqueness result from Theorem~\ref{thhhh:Corollary0002}, Proposition~\ref{prop:car-inerpol-space} and Proposition~\ref{lem:mbbL_analytic}.

	\begin{theorem} \label{th:Corollary0002}
	Let $0<\a<\b<1$ and  $B_0:= b(x)\mbL_K+L:\cC^{1+\b}_0(\mathbb{R}^{N-1})\to \cC^{\b-\alpha}_0(\mathbb{R}^{N-1})$ be given by \eqref{eq:bmbLKL}.  
		Let   $T>0 $ and    $\rho\in (0,\frac{1}{1+\a})$. Let  $u_0\in  \cC^{1+\b}_0(\mathbb{R}^{N-1})$ and  
		$f\in C^\rho([0,T], {\cC^{\b-\a}_0(\mathbb{R}^{N-1})})$ be such that  $f(0)-B_0 u_0\in { C^{\g_\rho-\a}(\R^{N-1})} $ with $\g_\rho=\b+\rho(1+\a)$.   Then the problem 
		\begin{equation}\label{eq:llkfp}
		\left\{
		\begin{array}{rll}
		u'(t)+B_0u(t)&=&f(t), \hspace{1cm}  t\in(0,T]\\
		u(0)&=&u_0
		\end{array}
		\right.
		\end{equation}
admits a unique   solution $u\in C^\rho([0,T], \cC^{1+\b}_0 )\cap C^{1+\rho}([0,T], \cC^{\b-\a}_0 )$. Moreover,
 there exists  $C_T=C(T,\rho,\k,N,\a,\g,\b)>0$ such that
			\begin{align*}
		\|  u- u_0\|_{C^\rho([0,T], C^{1+\b}) }&+ 	\| u-u_0\|_{  C^{1+\rho}([0,T], C^{\b-\a} )}\\
		&\leq C_T \Big(\|f- B_0 u_0\|_{C^\rho([0,T], C^{\b-\a}) }  +\|f(0)-B_0u_0 \|_{  C^{\g_\rho-\a} } \Big).		
\end{align*} 
In addition, $C_T\leq C_{T_0}$ for all $T\leq T_0$.
	\end{theorem}
	\begin{proof}
From  Theorem~\ref{thhhh:Corollary0002} and Proposition~\ref{lem:mbbL_analytic},  we get  a unique solution $u\in C^\rho([0,T], \cC^{1+\b}_0 )\cap C^{1+\rho}([0,T], \cC^{\b-\a}_0 )$  to \eqref{eq:llkfp}. Moreover,
				\begin{align}\label{eq:ssdfg}
	&	\| B_0( u- u_0)\|_{C^\rho([0,T], C^{\b-\a} ) }+ 	\| u-u_0\|_{  C^{1+\rho}([0,T],C^{\b-\a} )} \nonumber\\
		&\leq C_T \Big(\|f- B_0 u_0\|_{C^\rho([0,T], C^{\b-\a} )) }  +\|f(0)-B_0u_0 \|_{  \calD_{B_0}(\rho,\infty)}\Big).
\end{align} 
By Proposition \ref{prop:car-inerpol-space} we have  $\|f(0)-B_0u_0 \|_{  \calD_{B_0}(\rho,\infty)}\leq C' \|f(0)-B_0u_0 \|_{  C^{\g_\rho-\a} } $. On the other hand, since $B_0$ is closed, we obtain  
$$
\| B_0v\|_{ C^{\b-\a} }+ 	\|v\|_{   C^{\b-\a} } \geq C\left( \|v\|_{  C^{1+\b}   }+ 	\|v\|_{   C^{\b-\a} } \right),\quad\forall v\in C^{1+\beta}(\mathbb{R}^{N-1}).
$$
From this and \eqref{eq:ssdfg}, we get the result.
	\end{proof}
	\section{Regularity property of the nonlocal mean curvature operator}\label{Section2}
By changes of variables and the fact that $\cG(p)=-\displaystyle\int_{-p}^{p}\frac{d\t}{(1+\t^2)^{\frac{N+\a}{2}}}$  is odd, for $w\in C^{1,\g}_{loc}(\R^{N-1})$, with $\g>\a$, we have
	\begin{align*}
H(w)(x)&= P.V.\int_{\R^{N-1}}\frac{\cG(p_w({x},{y})) }{|{x}-{y}|^{N-1+\a}} d{y}\\
&=\frac{P.V.}{2} \int_{\R^{N-1}}\frac{\cG(p_w({x},{x-z}))+ \cG(p_w({x},{x+z}))}{|{z}|^{N-1+\a}} d{z}
\\&=\frac{P.V.}{2} \int_{\R^{N-1}}\frac{\cG(p_w({x},{x-z}))- \cG(-p_w({x},{x+z}))}{|{z}|^{N-1+\a}} d{z}.
\end{align*}
Therefore, by the fundamental theorem of calculus and the fact that $\cG'(p)=-2(1+p^2)^{\frac{-N-\a}{2}}$ is even, we get 
\begin{align}
\label{eq:exp-nmc-Bu}
H(w)(x)= \int_{\mathbb{R}^{N-1}}\frac{2w(x)-w(x-z)-w(x+z)}{|z|^{N+\alpha}}\cK_w(x,z)dz,
\end{align}
where
\begin{align}\label{Kernelu}
\cK_w(x,z)= \int_0^1\bigg(1+\left( \t p_w({x},{x-z})-(1-\t) p_w({x},{x+z}) \right)^2\bigg)^{-\frac{N+\alpha}{2}}d\tau .
\end{align}
For the remainder  of this section, it will be convenient to write
	\begin{align}
	H(w)(x)=   \int_{S^{N-2}}\int_0^\infty\frac{\d_e w(x,r,\th)}{r^{2+\alpha}}\mathcal{A}_w(x,r,\theta)drd\theta ,\label{Espression001}
	\end{align}	
	where $\d_e w(x,r,\theta)=2w(x)-w(x-r\theta)-w(x+r\theta)$, for all $(x,r,\theta)\in\mathbb{R}^{N-1}\times[0,\infty)\times S^{N-2}$  and 
	\begin{align}\label{ExtensionMap}
	\left.
	\begin{array}{rll}
	\mathcal{A}_{w}&:\mathbb{R}^{N-1}\times[0,\infty)\times S^{N-2}\rightarrow \mathbb{R}\\
	&(x,r,\theta)\mapsto \mathcal{A}_{w}	(x,r,\theta)=\cK_w(x,r\th), \;\;\; \cA_w(x,0,\th)= \bigg(1+(\n w(x)\cdot\th)^2\bigg)^{-\frac{N+\alpha}{2}}.
	\end{array}
	\right.
	\end{align}
	\begin{lemma}\label{lem:calA-diff}
		Let $k\in \N$ and $0<  \g<1$. Then there exists $C=C(k,N,\a,\gamma)>0$ such that  for all $u,w_1,\dots, w_k\in C^{1+\g}(\R^{N-1})$ , we have 
		$$
		\|\de^k_u \mathcal{A}_{u}[w_1,\dots, w_k] \|_{C^{\g  }(\R^{N-1}\times[0,\infty)\times S^{N-2})}\leq  C\left( 1+ \|\n u\|_{C^{\gamma}} \right)^C\prod_{i=1}^k\|\n w_i\|_{C^{\g} }.
	$$
	\end{lemma}
	\begin{proof}
		This follows from the definition of $\cA$.
	\end{proof}

	We have  the following lemma.
	\begin{lemma}\label{lem:F-smooth-Holder}
	Let $\g\in (\a,1)$.
Then the  map    $H:C^{1+\gamma}(\R^{N-1}) \to C^{\gamma-\a}(\R^{N-1})$   is of class $C^\infty$ and  for all $k\in\mathbb{N}, $ $u,w_1,\dots,w_k\in C^{1+\gamma}(\R^{N-1}),$
\be \label{eq:est-DkH}
		\|D^k H(u)[w_1,\dots, w_k]\|_{C^{\g-\a}}\leq C\left( 1+ \|\n u\|_{C^{\gamma} } \right)^C  \prod_{i=1}^k\|\n w_i\|_{C^{\gamma} } .
\ee
	with $C>0$ only depending on $k,N,\a,\gamma$.  Moreover     $H:\cC^{1+\g}_0(\R^{N-1})  \to {\cC^{\g-\a}_0} (\R^{N-1})$   is of class $C^\infty$ and  \eqref{eq:est-DkH} holds   for all $u,w_1,\dots,w_k\in \cC^{1+\g}_0(\R^{N-1})$.
	\end{lemma}

	\begin{proof}
		For all $k\in\mathbb{N}$,  $w_0, w_1,\dots,w_k\in C^{1+\gamma}(\R^{N-1}),$ we define 
		$$
		\calB_k^e(u)[w_0; w_1,\dots, w_k](x)=  \int_{S^{N-2}}\int_0^\infty \delta_e w_0(x,r,\theta) \de^k_u \mathcal{A}_{u}^e[w_1,\dots, w_k](x,r,\theta)r^{-2-\a }drd\theta.
		$$	
%
		%
		We now prove   that $H$ is of class $C^\infty$ and for all $k\in \N$,
		\begin{align}\label{eq:Dku-split}
		D^k H(u)[w_1\dots, w_k]&=  \calB^e_k(u)[ u; w_1,\dots, w_k]+ \sum_{j=1}^{k}\calB^e_{k-1}(u)[w_j; w_1,\dots, w_{j-1}, w_{j+1},\dots, w_k] .
		\end{align} 
		For this, it is enough to prove that $u\mapsto  \calB^e_k(u)[ u; w_1,\dots, w_k]$ is differentiable. By the linearity of the map $u\mapsto  \delta_e u$  it is also enough to prove the differentiability of   $u\mapsto  \calB^e_k(u)[ w_0; w_1,\dots, w_k]$   with   
		\begin{align}\label{eq:BkeBko-diff-est}
		\|D_u \calB^e_k(u)[w_0; w_1,\dots, w_k]\|_{C^{\gamma-\a } } \leq C  \left( 1+ \|\n u\|_{C^{\gamma}} \right)^C \prod_{i=0}^k\|\n w_i\|_{C^{\gamma} }  .
		\end{align}
		Let $v\in C^{1+\gamma}(\R^{N-1})$, with $\|\n v\|_{C^{\gamma} }\leq 1$.  We have 
		\begin{align*}
		&\calB^e_k(u+v)[ w_0; w_1,\dots, w_k](x)- \calB^e_k(u)[w_0; w_1,\dots, w_k](x)-\\
		&\int_{S^{N-2}}\int_0^\infty \delta_e w_0(x,r,\theta) \de^{k+1}_u \mathcal{A}_{u}^e[w_1,\dots, w_k,v](x,r,\theta)r^{-2-\a }\,drd\th\\
		&=\int_{S^{N-2}}\int_0^\infty r^{-2-\a}\delta_e w_0(x,r,\theta) \\&\times\int_0^1\left(\de^{k+1}_u \mathcal{A}_{u+\rho v }^e[w_1,\dots, w_k,v](x,r,\theta)-\de^{k+1}_u \mathcal{A}_{u}^e[w_1,\dots, w_k,v](x,r,\theta)\right) d\rho\,drd\th\\
		&=\int_{S^{N-2}}\int_0^\infty  r^{-2-\a }\delta_e w_0(x,r,\theta) \int_0^1\rho\int_0^1\de^{k+2}_u \mathcal{A}_{u+\rho\ov\rho  v }^e[w_1,\dots, w_k,v,v](x,r,\theta)  d\ov \rho d\rho\,drd\th
		\end{align*}
		Now by  Lemma \ref{SemigroupLemma001} and Lemma \ref{lem:calA-diff}, we get 
		\begin{align*}
		&\Big\|\calB^e_k(u+v)[ w_0, w_1,\dots, w_k]- \calB^e_k(u)[w_0, w_1,\dots, w_k]-\\
		&\int_{S^{N-2}}\int_0^\infty r^{-2-\a}\delta_e w_0(\cdot,r,\theta) \de^{k+1}_u \mathcal{A}_{u}^e[w_1,\dots, w_k,v](\cdot,r,\theta)\,drd\th \Big\|_{C^{\g-\a}}\\
		&\leq C  \left( 1+ \|\n u\|_{C^{\gamma-\a}} \right)^C \prod_{i=0}^k\|\n w_i\|_{C^{\gamma}}  \|\n v\|_{C^{\gamma-\a} }^2.
		\end{align*}
		From this, we deduce that  $u\mapsto  \calB^e_k(u)[ w_0; w_1,\dots, w_k]$ is  differentiable and \eqref{eq:BkeBko-diff-est} holds. \\
		Now the fact that $H:C^{1+\gamma}(\R^{N-1}) \to C^{\gamma-\a}(\R^{N-1})$ is of class $C^\infty$, follows easily by induction, thanks to \eqref{eq:Dku-split} and the estimates on $\calB^e_k$. Moreover the estimate on the derivative  $ H$ is an immediate consequence of those of $\calB^e_k$. \\

In view of \eqref{eq:Dku-split}, Lemma \ref{SemigroupLemma001} and Lemma \ref{lem:calA-diff}, the fact that   $H:\cC^{1+\g}_0 (\R^{N-1})\to \cC^{\g-\a}_0(\R^{N-1}) $ is of class $C^\infty$ follows similarly as above. In fact one can simply replace in the above argument $C^{1+\g}(\R^{N-1}) $ with $\cC^{1+\g}_0(\R^{N-1}) $ and $C^{\g-\a}(\R^{N-1})$ with ${\cC^{\g-\a}_0}(\R^{N-1}) $. 
	\end{proof}
%
We compute  next	 the explicit expression of $D H$.
\begin{lemma}\label{lem:DH-explicit}
Let $u\in C^{1+\gamma}_{loc}(\R^{N-1}) $, for some $\g\in (\a,1)$, with $\n u\in C^\g(\R^{N-1})$. Then  for all $w\in C^{1+\gamma} (\R^{N-1})$
\be
D H(u)[w](x)=- P.V.\int_{\R^{N-1}}\frac{w(x)-w(y)}{|x-y|^{N+\a}}\cG'(p_{u}(x,y))\, dy
\ee
and 
$$
\|D H(u)[w]\|_{C^{\g-\a}}\leq C (1+  \|\n u\|_{C^\g}) \|\n w\|_{C^{\g}},
$$
where $\cG'(p)=-2(1+p^2)^{-\frac{N+\a}{2}}.$
Here $C>0$, depends only on $N,\a,\b$ and $\g$.
If moreover  $\g>\b$ and $w\in \cC^{1+\b}_0 (\R^{N-1}) $, then   $D H(u)[w]\in {\cC^{\b-\a}_0} (\R^{N-1})$. 
\end{lemma}	
\begin{proof}
We consider the linear operator $M$ defined as
$$
M w(x):= - P.V.\int_{\R^{N-1}}\frac{w(x)-w(y)}{|x-y|^{N+\a}}\cG'(p_{u}(x,y))\, dy.
$$
We then have 
\begin{align*}
&H(u+w)(x)-H(u)(x)-Mw(x)  \\
%
%
&=P.V.\int_{\R^{N-1}} \int_0^1\frac{[\cG'(p_{u+\t w}(x,y))-\cG'(p_u(x,y))  ]  p_w(x,y)  }{|x-y|^{N-1+\a}} d\t  \, dx.
\end{align*}
We then get 
\begin{align}\label{eq:dif-quo-H}
&\G(x):=H(u+w)(x)-H(u)(x)-Mw(x)  =P.V.\int_{\R^{N-1}}\frac{  (p_w(x,x-y))^2 \mu(x,|y|, y/|y|)}{|y|^{N-1+\a}} \, dx,
\end{align}
where 
$$
\mu(x,r,\th):=2 \int_0^1\t\int_0^1\cG''(p_{u+\t \varrho w}(x,x-r\th))d\t d\varrho 
$$
and 
$$
\mu(x,0,\th):=2 \int_0^1\t\int_0^1\cG''(\n u(x)\cdot \th+\t\varrho \n w(x)\cdot \th) d\t d\varrho .
$$
Using polar coordinates and recalling \eqref{eq:def-deeo}, we can write
\begin{align*}
&\G(x):=P.V.\int_{S^{N-2}}\int_0^\infty r^{-2-\a}(w(x)-w(x-r\th)) \int_0^1 \n w(x-r \t \th)\cdot \th\, d\t\mu(x, r,\th)\, dr d\th\\
&=  \int_{S^{N-2}}\int_0^\infty r^{-2-\a}\d_ow(x,r,\th)  \int_0^1( \n w(x-r \t \th)-\n w(x)) \cdot\th\, d\t\mu(x, r,\th)\, dr d\th\\
&+ \int_{S^{N-2}}\int_0^\infty r^{-2-\a}\d_ow(x,r,\th)  \n w(x) \cdot\th (\mu(x, r,\th)-\mu(x, 0,\th))\, dr d\th \\
&+   \int_{S^{N-2}}\int_0^\infty r^{-2-\a}\d_ew(x,r,\th) \n w(x) \cdot\th  \mu(x, 0,\th)\, dr d\th,
\end{align*}
where we used that  $\mu(x,0,-\th)= -\mu(x,0,\th)$, from the oddness of $\cG''$. Since 
$$
\|\mu\|_{C^\g(\R^{N-1}\times [0,\infty)\times S^{N-2})}\leq C (1+ \|\n w\|_{C^\g}+ \|\n u\|_{C^\g}),
$$
we can thus apply Lemma \ref{SemigroupLemma001}$(i)$ to deduce that
$$
\|\G\|_{C^{\g-\a}}\leq C (1+ \|\n w\|_{C^\g}+ \|\n u\|_{C^\g}) \|\n w\|_{C^{\g}}^2.
$$
Recalling \eqref{eq:dif-quo-H}, this completes the  proof  of the first statement of the lemma. The second statement follows  from     Lemma \ref{SemigroupLemma001}$(iii)$. 
\end{proof}
We deduce the following important result for the weighted fractional mean curvature operator $\cH$ defined in \eqref{Weighted:MCoperator}.
		\begin{corollary}\label{lem:F-smoothXY}
	
		The map $\cH:\cC^{1+\b}_0(\R^{N-1})\to \cC^{\b-\alpha}_0(\R^{N-1})$  is of class  $C^\infty$. Moreover,  	letting $u_0\in C^{1+\gamma}_{loc}(\R^{N-1}) $ for  some $\gamma\in(\b,1)$, with $\n u_0\in C^\g(\R^{N-1})$,  
		  then the  map 
		  $$
		  F:\cC^{1+\b}_0(\R^{N-1})\to \cC^{\b-\alpha}_0(\R^{N-1}), \qquad F(u)=D\cH(u_0)[u]-\cH(u)
		  $$
		   is of class $C^\infty$ and  for all $k\in \N$,
	$$
		\|D^k F(u)\|_{C^{\b -\alpha}} \leq C\left( 1+ \|\n u\|_{C^\b}+ \|\n u_0\|_{C^\b}\right)^C,
	$$
		with $C=C(k,N,\a,\b )>1$. 
	\end{corollary}
	
	\begin{proof}
	Recall that $ \cH(u)=Q(u)H(u)$ with $Q(u)=\sqrt{1+|\n u|^2 }$ and we observe that $Q\in C ^\infty(\cC^{1+\b}_0(\R^{N-1}),\cC^{\b-\alpha}_0(\R^{N-1})) $ and $DQ(u_0)\in C ^\infty(\cC^{1+\b}_0(\R^{N-1}),\cC^{\b-\alpha}_0(\R^{N-1})) $.  The conclusion then follows from 
  Lemma \ref{lem:F-smooth-Holder} and Lemma \ref{lem:DH-explicit}.
	\end{proof}

\subsection{Regularity properties and linearization of the  weighted fractional mean curvature operator}
We recall that 
$$
	\cH(u):= Q(u)H(u), \qquad Q(u):=\sqrt{1+|\nabla u(x)|^2}.
$$
Our aim is to prove that, provided $u_0\in C^{1+\g}_{loc}(\R^{N-1})$ with $\g\in (\b,1)$ and $\n u_0\in C^{\g}_{loc}(\R^{N-1})$, we have that  	
	$D \cH(u_0) : \cC^{1+\b}_0(\R^{N-1})\to \cC^{\b-\alpha}_0(\R^{N-1})$ generates a strongly continuous  analytic semigroup. 
	We observe that   for all $w\in C^{1+\b}(\R^{N-1})$,  
\be\label{eq:decomp-GNMC}
D\cH(u_0)[w] :=Q(u_0) \ti{L}_1 w+\ti{L}_2w,
\ee
where 
\begin{align}\label{eq:tiL_1}
\ti{L}_1 w(x)&:=-   P.V.\int_{\R^{N-1}}\frac{w(x)-w(y)}{|x-y|^{N+\a}}\cG'(p_{u_0}(x,y))\, dy\ \nonumber\\
&=  P.V.\int_{\R^{N-1}}\frac{w(x)-w(y)}{|x-y|^{N+\a}}  (1+(p_{u_0}(x,y))^2)^{\frac{N+\a}{2}}\,dy
\end{align}
 and 
$$
\ti{L}_2w(x):=H(u_0)(x) DQ(u_0)[w] (x)= H(u_0)(x) \frac{\n u_0(x)\cdot \n w(x)}{Q(u_0)(x)}.
$$
For the following, we recall the set 
\be\label{eq:defcO-cg}
\calO_{\nu}^\g:=\{ u\in C^{1+\gamma}_{loc}(\mathbb{R}^{N-1})\,:\, \|\n u\|_{C^\g}\leq \nu \},
\ee
for $\nu>0$.
We have the following result.
	\begin{lemma}\label{NonlocalOperator--1}
Let $\g\in(\alpha,1) $. Then the  nonlocal mean curvature operator $H$ satisfies the following properties.
\begin{enumerate}
\item[(i)]    There exists  $C'=C'(N,\alpha,\g)>0$ such that  for all $u_0\in \calO_{\nu}^\g$,
\be\label{eq:est-Cnmc-nu}
 \| \cH(u_0)\|_{ C^{\gamma-\alpha} }+ \| H(u_0)\|_{ C^{\gamma-\alpha} }\leq C'  \|\n u_0\|_{ C^{\gamma} } .
\ee
\item[$(ii)$] If  $u\in \cC^{1+\g}_0(\R^{N-1})$ then  $\cH(u), H(u)\in \cC^{\g-\a}_0(\R^{N-1})$. 
\item[(iii)]  There exists $C=C(N,\a,\g,\nu)$ such that if $u\in \calO_{\nu}^{\g-\a}\cap C^{1+\b}_{loc}(\R^{N-1})\cap L^\infty(\R^{N-1})$, for some $\b>\a$, and satisfies $\cH(u)\in C^{\g-\a}(\R^{N-1})$, then $u\in C^{1+\g} (\R^{N-1})$  and
\be \label{eq:graph-norm-equiv-nmc}
C		\|\n u\|_{ C^{\gamma} }\leq \| \cH(u)\|_{ C^{\gamma-\alpha} }+ \|  u\|_{ L^{\infty} } .
\ee
\end{enumerate}
	\end{lemma}
	
	\begin{proof}
In view of  \eqref{Espression001} and since $\cH(u_0)=\sqrt{1+|\n u_0|^2}H(u_0)$,  we can apply Lemma \ref{SemigroupLemma001} to get $(i)$ and $(ii)$.\\
For $(iii)$, we let $f:=\cH(u)$ and  recalling  \eqref{eq:exp-nmc-Bu}, we then have
$$
 \int_{\mathbb{R}^{N-1}}\frac{2u(x)-u(x-y)-u(x+y)}{|y|^{N+\alpha}}\ti \cK_u(x,y)dy=f(x)\qquad\textrm{ for all $x\in \R^{N-1}$,}
$$
where $\ti \cK_u(x,y)=\sqrt{1+|\n u(x)|^2}\cK_u(x,y)$.
Since $u\in \cO_{\nu}^{\g-\a}$, we have $$\sup_{x,y \in \R^{N-1}}|\ti \cK_u(x+h,y)-\ti \cK_u(x,y)|\leq C(\nu,N,\a,\g)|h|^{\g-\a}\qquad\textrm{ for all $h\in \R^{N-1}$}.$$  In addition $(1+4\nu^2)^{\frac{-N-\a}{2}}\leq \ti \cK_u(x,y)\leq 1 $, for all $x,y\in \R^{N-1}$.
Therefore, applying \cite[Theorem 1.2]{bass2009regularity}, we obtain
$$
C		\|\n u\|_{ C^{\gamma} }\leq \| f\|_{ C^{\gamma-\alpha} }+ \|  u\|_{ L^{\infty} } .
$$
and the proof is complete.
	\end{proof}
We now have the following result.

\begin{lemma}\label{lemm:SemigropeCorollary}
Let $w, v \in  \calO_{\nu}^\g$, for some $\gamma\in(\a,1)$. Then  the linear  operator  
\be \label{eq:def-main-op-B}
B[w,v]: C^{1+\gamma}(\mathbb{R}^{N-1})\mapsto C^{\gamma-\a}(\mathbb{R}^{N-1}), \qquad B[w,v] u:= \int_{0}^1D\calH(\varrho w+(1-\varrho)v)[u]d\varrho, 
\ee
 satisfies, for all $u\in C^{1+\g}_{loc}(\R^{N-1})$, with $\n u\in C^\g(\R^{N-1})$,
	\be\label{eq:graph-norm-lin-nmc-mean}
	  \|B[w,v] u\|_{ C^{\gamma-\alpha} } \leq C \|\n u\|_{ C^{\gamma} } ,
\ee
and 
\be\label{eq:upper-grad}
C'		\|\n u\|_{ C^{\gamma}}\leq \| B[w,v] u\|_{ C^{\gamma-\alpha} }+ \| \n  u\|_{ L^{\infty}} , 
\ee
for some constants $C,C'$   depending only on $N,\g,\a$ and $\nu$.\\
Moreover for all  $\b\in(\a,\g)$,  the operator  $B:=B[w,v]$, with domain $\cC^{1+\b}_0(\R^{N-1})$, 
 is an infinitesimal generator of a strongly  continuous  analytic semigroup $\{e^{t B}:t\geq 0\}$ on $\cC^{\b-\alpha}_0(\R^{N-1})$.\\
In addition,  there exists $m=m(N,\a,\b,\g,\nu)\in \R$ such that  for $k\in \mathbb{N}$,   
	there exists $M_k=M_k(N,\a,\b,\g,\nu,m)>0$,   such that
	\begin{align}\label{eq:analytic-sm}
	\|t^k B^ke^{tB }\|_{\mathcal{L}(\cC^{\b-\alpha}_0(\R^{N-1}))}\leq M_k e^{mt},\hspace{.5cm}\mbox{for all  $t>0$.}
	\end{align}
\end{lemma}
\begin{proof}
%
Thanks  to \eqref{eq:decomp-GNMC}, we can write 
\be \label{eq:B-integrro-diff}
B[w,v]=Q(w)B_1+  B_2,
\ee
where 
 $Q(w)=\sqrt{1+|\n w|^2}$ and 
$$
B_1 u=- \int_{\R^{N-1}}\frac{u(x)-u(y)}{|x-y|^{N+\a}}   \int_{-1}^1\cG'(\varrho p_{   w}(x,y)+(1-\varrho) p_{v}(x,y))  d\varrho dy,
$$
$$
B_2 u:=H(v) \int_0^1DQ(\varrho w+(1-\varrho) v)[ u] d\varrho   .
$$
Recall that $\cG'(p)=-(1+p^2)^{-\frac{N+\a}{2}}$ and  
$$
p_U(x,y)=-\frac{U(x)-U(y)}{|x-y|}=-\int_0^1\n U(\t (x-y)+y)\cdot \frac{x-y}{|x-y|}\, d\t.
$$
Therefore since  $p_U(x,x-r\th)=-\int_0^1\n U(x-r\t\th)\cdot \th\, d\t$,  we see that   $B_1$ satisfies the properties in Assumptions \ref{ass:L},  with $K(x,y):=- \int_{-1}^1\cG'(\varrho p_{   w}(x,y)+(1-\varrho) p_{v}(x,y))  d\varrho$ and 
 $\k$, depending only on $N,\a,\g, \b $ and $\nu$.  
 Applying Lemma \ref{SemigroupLemma001}, we obtain \eqref{eq:graph-norm-lin-nmc-mean}. 
On the other hand by Lemma \ref{lem:mbL-graph-norm},  we have 
\be \label{eq:omicr}
C'		\|\n u\|_{ C^{\gamma} }\leq \| B_1 u\|_{ C^{\gamma-\alpha} }+ \| \n  u\|_{ L^{\infty} }.
\ee
From this and  \eqref{eq:interpol}, we get 
\be \label{eq:B_eperturb-B_1}
 \|B_2 u\|_{ C^{\gamma-\alpha}}  \leq  \e \|B_1 u\|_{ C^{\gamma -\a}}+ c_\e\| \n u\|_{L^\infty}.
\ee
Combining  this with  \eqref{eq:omicr}, we obtain  \eqref{eq:upper-grad}.\\

We now  let $\b\in (\a,\g)$ and $u\in\cC^{1+\b}_0(\R^{N-1})$ so that $B_2u \in \cC^{\b-\a}_0(\R^{N-1})$. Then   \eqref{eq:interpol} and \eqref{eq:B_eperturb-B_1} imply that 
$$
 \|B_2 u\|_{ C^{\b-\alpha}}  \leq  \e \|B_1 u\|_{ C^{\b-\a} }+ c_\e\|  u\|_{C^{\b-\a}}.
$$
Moreover,  
$$
A_K^e(x,0,\th)=-2 \int_{-1}^1 (1+(\varrho \n w(x)\cdot\th+(1-\varrho) \n v(x)\cdot\th)^2)^{-\frac{N+\a}{2}}  d\varrho
$$
and clearly satisfies $  \sup_{\th\in S^{N-2}}\|\n^k_\th A_K^e(\cdot,0,\th)\|_{C^\g}\leq C_k(N,\a,\g,\nu). $
We can thus   apply  Lemma   \ref{lem:mbbL_analytic}, to deduce that  $B[w,v]:\cC^{1+\b}_0(\R^{N-1})\to \cC^{\b-\alpha}_0(\R^{N-1})$ generates a strongly continuous analytic semigroup.  
Since $B[w,v]$ is sectorial,  \eqref{eq:analytic-sm} follows from  \cite[Proposition 2.1.1]{lunardi2012analytic}.\\

\end{proof}
As	 a consequence of the above lemma, we the following result.
\begin{corollary}\label{corr:SemigropeCorollary}
Let  $\b\in (\a,1)$ and $u_0\in  \calO_{\nu}^\g$, for some $\gamma\in(\b,1)$. Then  the linear  operator  
\be \label{eq:calL_0-analy}
\calL_0: \cC^{1+\b}_0(\R^{N-1}) \mapsto \cC^{\b-\a}_0(\R^{N-1}) , \qquad \calL_0 u:=  D\calH(u_0)[u], 
\ee
is an infinitesimal generator of  a strongly continuous analytic semigroup $\{e^{t \cL_0}:t\geq 0\}$ on $\cC^{\b-\alpha}_0(\R^{N-1})$.\\
   Moreover, for   $T>0 $, $\rho\in (0,\frac{1}{1+\a})$,   $u_0\in  \cC^{1+\b}_0(\mathbb{R}^{N-1})$ and  
		$f\in C^\rho([0,T], {\cC^{\b-\a}_0(\mathbb{R}^{N-1})})$  such that  $f(0)-\cL_0 u_0\in { C^{\g_\rho-\a}(\R^{N-1})}, $ with $\g_\rho=\b+\rho(1+\a)$,     the initial value problem 
		\begin{equation*} 
		\left\{
		\begin{array}{rll}
		u'(t)+\cL_0u(t)&=&f(t), \hspace{1cm}  t\in(0,T]\\
		u(0)&=&u_0
		\end{array}
		\right.
		\end{equation*}
admits a unique   solution $u\in C^\rho([0,T], \cC^{1+\b}_0(\mathbb{R}^{N-1}))\cap C^{1+\rho}([0,T], \cC^{\b-\a}_0(\mathbb{R}^{N-1}))$. In addition,
 there exists  $C_T=C(T,\rho,\nu,N,\a,\g,\b)>0$ such that
			\begin{align*}
		\|  u- u_0\|_{C^\rho([0,T], C^{1+\b}) }&+ 	\| u-u_0\|_{  C^{1+\rho}([0,T], C^{\b-\a} )}\\
		&\leq C_T \Big(\|f- \cL_0 u_0\|_{C^\rho([0,T], C^{\b-\a}) }  +\|f(0)-\cL_0u_0 \|_{  C^{\g_\rho-\a} } \Big),
\end{align*} 
with, $C_T\leq C_{T_0}$ for all $T\leq T_0$.
\end{corollary}
\begin{proof}
Since $B[u_0,u_0]=D\calH(u_0)$, the result follows from Lemma \ref{lemm:SemigropeCorollary}, \eqref{eq:B-integrro-diff} and Theorem \ref{th:Corollary0002}. 
\end{proof}

	\begin{remark} \label{rem:DB_0-rho}
Let  $\rho\in (0, \frac{1}{1+\a})$ and define $\g_\rho=\b+\rho(1+\a)$ and suppose that $\g_\rho-\a\not=1$.  Then by Proposition \ref{prop:car-inerpol-space}, 
if $B:=B[w,v]$ is given by \eqref{eq:def-main-op-B}   then thanks to \eqref{eq:graph-norm-lin-nmc-mean}, $B u\in  \calD_{B}(\rho,\infty)$ as soon as  $\n u\in C^{\g_\rho}(\R^{N-1})\cap \cC^{\b-\a}_0(\R^{N-1})$. \\
As a  consequence, if  $ u_0\in  \cC^{1+\b}_0(\R^{N-1})$, for some $\b>\a$, then by  \eqref{eq:graph-norm-equiv-nmc},   $\calH(u_0)\in  \calD_{\calL_0}(\rho,\infty)$  if and only if  $\n u_0\in C^{\g_\rho}(\R^{N-1}) \cap \cC^{\b-\a}_0(\R^{N-1})$, where $\calL_0:=D\calH(u_0)$.
	\end{remark}
Our next result shows that the operator $B[w,v]$ (defined in \eqref{eq:def-main-op-B}) satisfies the maximum principle. 
\begin{lemma}\label{lem:explicit-Bwv}
Let  $\g>\a$ and  $ w, v\in C([0,T],\cO^\g_\nu) $. Then  for all $u\in \cO^\g_\nu $, 
\begin{align*}
B[w(t),v(t)]u(x)= P.V.\int_{\R^{N-1}}\frac{u(x)-u(x+y)}{|y|^{N+\a}} \mu(t,x,|y|, y/|y|) \, dy+ V(t,x)\cdot \n u(x), \nonumber
\end{align*}
where, for  $(x,r,\th)\in \R^{N-1}\times [0,\infty)\times S^{N-2}$ we have  $p_w(x,x-r\th)=-\int_0^1\n w(x-r\t\th)\cdot\th d\t$, 
$$
\mu (t,x,r,\th):=-\int_{-1}^1\cG'(\varrho p_{w(t) }(x,x-r\th )+(1-\varrho) p_{v(t)}(x,x-r\th)) Q(\varrho w(t)+(1-\varrho)v(t))(x)  d\varrho,
$$
 and
$$
V(t,x):=\int_{-1}^1 \frac{\varrho \n w(t,x)+(1-\varrho) \n v(t,x) }{ Q(\varrho w(t)+(1-\varrho) v(t))(x)}  H(\varrho w(t)+(1-\varrho)v(t))(x)  d\varrho.
$$ 
Moreover   $ \mu(t,x,0, \th)= \mu(t,x,0, -\th) $ and  $\mu \in C([0,T],C^\g( \R^{N-1}\times [0,\infty)\times S^{N-2}))$ and $V\in C_b([0,T]\times  \R^{N-1}).$
	\end{lemma}
	\begin{proof}
	By the fundamental theorem of calculus, we can write $ p_w(x,x-r\th)=\frac{w(x-r\th)-w(x)}{r}=-\int_0^1\n w(x-r\t\th)\cdot\th d\t$.
Using \eqref{eq:B-integrro-diff}, we get the expression of 	 $B[w(t),v(t)]$.  Moreover  from the evenness of $\cG'$  we have  $ \mu(t,x,0, \th)= \mu(t,x,0, -\th) $.  The regularity of $V$ is a consequence of \eqref{eq:est-Cnmc-nu}.
	\end{proof}
As a consequence,  of  Lemma \ref{lem:explicit-Bwv} and Proposition \ref{prop:mmp}, we have the following result. 
\begin{corollary}\label{cor:mmp-lnmc}
Under the assumptions of Lemma \ref{lem:explicit-Bwv}, let $u \in L^\infty([0,T], C^{\b}(\R^{N-1}))$, with $u(\cdot,x)\in C^1([0,T])$ and $u(t,\cdot)\in C^{1+\b}_{loc}(\R^{N-1})$,  satisfy
$$
\partial_tu+B[w(t),v(t)] u\leq 0\qquad\textrm{  in $ [0,T]\times \R^{N-1}$.}
$$
Then $\sup_{[0,T]\times \R^{N-1}}u=\sup_{x\in \R^{N-1}} u(0,x)$.
\end{corollary}

 	\section{Proof of  the main results  }\label{Section5}
	For $0<\rho<\frac{1}{1+\a}$, $\b\in (\a,1)$ and $T>0$, we define for  the following  Banach space 
 \be \label{eq:def-E_T}
 E_T=C^\rho([0,T],\cC^{1+\b}_0(\R^{N-1}))\cap C^{1+\rho}([0,T],\cC^{\b-\alpha}_0(\R^{N-1}))
 \ee
 endowed with  the norm $\|\cdot\|_{E_T}=\|\cdot \|_{C^\rho([0,T],C^{1+\b} )}+\|\cdot \|_{C^{1+\rho}([0,T],C^{\b-\alpha} )} $. 	
	\begin{lemma} \label{Lemma001}	Let  $u_0 \in  C^{1+\g}_{loc}(\R^{N-1})\cap \cC^{1+\b}_0 (\R^{N-1})$, for some $\g>\b$, with $\n u_0\in C^\g(\R^{N-1})$. 
For $R>0$, 	we define
	\be \label{eq:def-E_T0R}
E_{T,R}:=\left\{u\in E_T\,:\,   u(0)=u_0\,,\,  \|\n u-\n u_0\|_{C^\rho([0,T],C^{ \b}(\R^{N-1}) )}   \leq R \right\}	.
\ee
		Then,   there exists  $C= C( N,\a, \b)>1$  with the property that  for    every $u,v\in  E_{T,R}  $  we have
		\begin{align*}
		\|F(u)-F(v)\|_{C^\rho([0,T],{C^{\b-\alpha}})}\leq C R(1+R+\|\n u_0\|_{C^{\b} })^C   T^\rho(1+T^\rho)   \|u-v\|_{E_T},
		\end{align*}



where $F: \cC^{1+\b}_0(\R^{N-1})\to \cC^{\b-\a}_0(\R^{N-1})$ is given by  $F(u)=-\cH(u)+D\cH(u_0) u$.
	\end{lemma}
	
	\begin{proof}


          
          Let $T, R>0$, $u,v\in  E_{T,R}$ and $w=u-v$. Then we have 
		\begin{align*}
&		F(u)-F(v)=- \int_0^1(D \cH (\l u+(1-\l)v) - D\cH(u_0) ) [w]\, d\l\\
		&= - \int_0^1\l \int_0^1D^2 \cH (\t(\l u+(1-\l)v)+(1-\t) u_0)  [\l (u-u_0)+(1-\l)(v-u_0), w]\, d \t d\l.
		\end{align*}
Next, for $t\in [0,T]$, $\lambda, \tau \in [0,1]$ we define  $L_{\lambda,\tau}(t) \in \cC^{\b-\alpha}_0(\R^{N-1})$ by 
$$
L_{\lambda,\tau}(t):=   D^2 \cH (\t(\l u(t)+(1-\l)v(t))+(1-\t) u_0).
$$
We observe that, by  Corollary \ref{lem:F-smoothXY}, for all $s,t\in [0,T]$, 
\be\label{eq:LtLs}
\|L_{\lambda,\tau}(t)-L_{\lambda,\tau}(s)\|_{C^{\b -\alpha}} \leq CR (1+R+\|\n u_0\|_{C^{\b} })^C |s-t|^\rho 
\ee
and 
\be\label{eq:Lt}
\|L_{\lambda,\tau}(t) \|_{C^{\b -\alpha}} \leq CR (1+R+\|\n u_0\|_{C^{\b} })^C .
\ee
We write
\begin{align*}
&[F(u)-F(v) ](t)-[F(u)-F(v) ](s)\\
  &= \int_0^1\l \int_0^1\Bigl(L_{\lambda,\tau}(s)[\l (u(s)-u_0)+(1-\l)(v(s)-u_0), w(s)]\\
  &\qquad \qquad \;\: - L_{\lambda,\tau}(t)[\l (u(t)-u_0)+(1-\l)(v(t)-u_0), w(t)]\Bigr)d \t d\l\\ 
  &= \int_0^1\l \int_0^1\Bigl([L_{\lambda,\tau}(s)-L_{\lambda,\tau}(t) ] [\l (u(s)-u_0)+(1-\l)(v(s)-u_0), w(s)]\, \\
&\qquad \qquad \;\: +L_{\lambda,\tau}(t)[\l (u(s)-u_0)+(1-\l)(v(s)-u_0), w(s)-w(t)]\\
&\qquad \qquad  \;\: +L_{\lambda,\tau}(t)[\l (u(s)-u(t))+(1-\l)(v(s)-v(t)), w(t)]\Bigr)d \t \,d\l.
\end{align*}
Using \eqref{eq:LtLs}, \eqref{eq:Lt} and the fact that $u,v\in E_{T,R}$, we then get 
\begin{align}
  &\| [F(u)-F(v) ](t)-[F(u)-F(v) ](s)\|_{C^{\b-\alpha}} \nonumber \\
  &\qquad \qquad \qquad \leq C R(1+R+\|\n u_0\|_{C^{\b} })^C (t^\rho+s^\rho) |t-s|^\rho \|w\|_{E_T}.\label{eq:Fuv-holder}
\end{align}
Using the fact that $[F(u)-F(v)](0)=0$, $w\in E_T$ and taking $s=0$ in \eqref{eq:Fuv-holder}, we get 
\be \label{eq:Fuv-L-infty-rho}
\|F(u)-F(v)\|_{L^\infty([0,T],{\cC^{\b-\alpha}_0})}\leq CR (1+R+\|\n u_0\|_{C^{\b} })^C   T^{2\rho} \|w\|_{E_T}.
\ee
	The result follows from \eqref{eq:Fuv-holder} and  \eqref{eq:Fuv-L-infty-rho}.

	\end{proof} 	
 	
 		\begin{proof}[Proof of Theorem \ref{MainResult}]
	Let $u_0\in {\cC^{1+\b}_0(\R^{N-1})}\cap {C^{\b+\rho(1+\a)}_{loc} (\R^{N-1})}  $ with $\|\n u_0 \|_{C^{\b+\rho(1+\a)} } \leq \nu$. Then  by Corollary \ref{corr:SemigropeCorollary}, the operator $\cL_0:=D\cH(u_0)$  is a generator of strongly continuous analytic semigroup.   
  Now,  to solve  problem  \eqref{CauchyProblem}, we rewrite it  as
	\begin{align}
\left\{
\begin{array}{rll}
\de_tu+\cL_0u&=F(u)&\hspace{.5cm}\mbox{in }\;[0,T]\times 	\mathbb{R}^{N-1}	\\
u(0)&=u_0&\hspace{.5cm}\mbox{in }\;\mathbb{R}^{N-1}
\end{array}
\right.
\end{align}	
and use a fixed point argument, with $F(u):=-\cH(u)+D\cH(u_0) [u]=-\cH(u)+\cL_0 u$. 
Let  $u\in E_T$ and recall \eqref{eq:def-E_T} for the definition of $E_T$. Note that by \eqref{eq:graph-norm-lin-nmc-mean} and Lemma\ref{NonlocalOperator--1},   $F(u)(0)-\cL_0u(0)=-\cH(u_0) \in  \cC^{\b-\a}_0(\R^{N-1}) \cap C^{ \b+\rho(1+\a)-\a}(\R^{N-1})$. Hence    by Corollary  \ref{corr:SemigropeCorollary},  there exists a  unique   function  $\Phi(u)\in E_T$  satisfying
\begin{align}\label{InitialProblem02}
\left\{
\begin{array}{rll}
\de_t \Phi(u)+\cL_0\Phi(u)&=F(u)&\hspace{.5cm}\mbox{in }\;[0,T] 	\times\mathbb{R}^{N-1}	\\
\Phi(u)(0)&=u_0&\hphantom{.5cm}\mbox{in }\;\mathbb{R}^{N-1}.
\end{array}
\right.
\end{align}	
We define 
$$
		\calE_{T,R}:=\left\{u\in E_T\,:\, u(0)=u_0\,,\, \|u-u_0\|_{E_T}\leq R\right\}.
		$$
Clearly a fixed point of the map $\Phi: E_T\to E_T$ in the   set $\calE_{T,R}$ will be a solution to \eqref{CauchyProblem}.\\

	We claim that, provided $\|\n u_0 \|_{C^{\b+\rho(1+\a)}(\R^{N-1})} \leq \nu$,   there exist a large constant  $R>0$ and a small constant  $T>0$, both depending only on $ N,\a,\b,\rho,\g$ and $ \nu $,   such that 
		\begin{enumerate}
			\item[$(i)$]$\Phi(\calE_{T,R})\subset \calE_{T,R}$ .
			\item[$(ii)$]  $\Phi$ is a contraction on $\calE_{T,R}.$
		\end{enumerate}
To prove  $(i)$,  we use 	 Corollary  \ref{corr:SemigropeCorollary}, \eqref{eq:est-Cnmc-nu}  and   Lemma \ref{Lemma001}, to get    $C=C(\rho,N,\a,\b, \g,\nu )>1$    such that  for all $T\in (0,1)$,
		\begin{align*}
		  	\|\Phi(u)-u_0\|_{E_T} &\leq C \left(  \|F(u)-\cL_0 u_0 \|_{C^\rho([0,T],{C^{\b-\alpha}})}+ \|\cH( u_0)\|_{C^{\b+\rho(1+\a)-\a} } \right)\\
		&\leq C \left(  \|F(u)-F(u_0)\|_{C^\rho([0,T],{C^{\b-\alpha}})}+ \|\n u_0 \|_{C^{\b+\rho(1+\a)} } \right) \\
		&\leq  C \left(  T^{\rho}(1+T^{\rho}) R(1+ R+\|\n u_0\|_{C^\b} )^C+ \|\n u_0 \|_{C^{\b+\rho(1+\a)} }\right).
		\end{align*}
We can thus let  $R=2\nu C$ and choose $T >0$ small, so that  $\Phi(u)\in  \calE_{T,R}$.\\

To prove $(ii)$, we let  $u,v\in 	 \calE_{T,R}$. Then  $W:=\Phi(u)-\Phi(v)$ satisfies
$$
\left\{
\begin{array}{rll}
\de_tW+\cL_0 W&=F(u)-F(v)&\hspace{.5cm}\mbox{in }\;[0,T] 	\times\mathbb{R}^{N-1}	\\
W(0)&=0&\hphantom{.5cm}\mbox{in }\;\mathbb{R}^{N-1}.
\end{array}
\right.
$$
Hence by 	Corollary  \ref{corr:SemigropeCorollary}  and   Lemma \ref{Lemma001}, there exists  $C=C( \rho,N,\a,\b, \g, \nu)>1$ such that for  $T\in (0,1)$,
		\begin{align*}
		\|\Phi(u)-\Phi(v)\|_{E_T} &\leq C \|F(u)-F(v)\|_{C^\rho([0,T],{C^{\b-\alpha}})}  \leq C  T^{\rho} R (1+ R+\|\n u_0\|_{C^\b} )^C \|u-v\|_{E_T}.
		\end{align*}
Now  decreasing $T$, if necessary, we obtain $\|\Phi(u)-\Phi(v)\|_{E_T}\leq  \frac{1}{2}\|u-v\|_{E_T}$, which is $(ii)$.
		This ends the proof of the claim. \\
We can thus apply the Banach fixed point theorem on $\calE_{T,R}$ to get a unique fixed point $u\in \calE_{T,R}$ of $\Phi$.  The proof of  the theorem  is thus finished.  

\end{proof}

%
Theorem   \ref{MainResult1} and Theorem \ref{MainResult1-part-2}  are  consequences of the following result which deals with bounded functions.
\begin{theorem}\label{MainResult1-bnd}
Under the assumptions of Theorem \ref{MainResult},  we have the following properties.
\begin{enumerate}
\item[(i)]    For all $\b'\in (\a,\b)$, $\rho \in (0,\frac{\a}{1+\a}]$,   $k\in \N $ and $t\in (0,T]$, we have $u(t)\in \cC^{k+1+\b'}_0(\R^{N-1}) $ and  there exists $C_k>0$ only depending on $k,\rho,\a,\b,\g,N,\nu,\b'$ and $T$  such that 
\be \label{eq:main-2}
	 	  \|t^k \n u \|_{C^{\rho}([0,T],C^{k+ \b'} )}  \leq C_k  .
\ee
\item[(ii)] If, in addition, $ \n u_0\in C^{1+\g_\rho}(\R^{N-1}) $  for $\rho\in (0,\frac{1}{1+\a})$, then  there exists $C>0$ only depending on $\rho,\a,\b,\g,N,\nu,T$ and $\b'$     such that 
\be \label{eq:main-3}
	 	  \|\n  u \|_{C^{\rho}([0,T],C^{1+ \b'} )}  \leq C  \|\n u_0\|_{ C^{1+\g_\rho} }  ,
\ee
where $\g_\rho=\b+\rho(1+\a)$.

\end{enumerate}
\end{theorem}
	\begin{proof}
We define   $\t_hu(t,x)=\frac{u(t,x+h)-u(t,x)}{|h|}$, for $0<|h|<1$. From \eqref{CauchyProblem},  we deduce that 
\be \label{InitialProblem02-h}
\left\{
\begin{array}{rll}
\de_t \t_hu +L_h(t) \t_hu&=0 &\hspace{.5cm}\mbox{in }\;[0,T] \times\mathbb{R}^{N-1}	\\
\t_hu (0)&=\t_hu_0 &\hspace{.5cm}\mbox{in }\;\mathbb{R}^{N-1},
\end{array}
\right.
\ee
where for all $a\in \R^{N-1}$, we define 
\be \label{eq:defLat}
L_a(t) v:=\int_0^1D\calH(\varrho u(t,\cdot+a)+(1-\varrho)u(t))[v]\, d\varrho.
\ee
We observe, from Lemma \ref{lem:explicit-Bwv},  that
 \begin{align}
L_a(t)u(x)&=P.V.\int_{\R^{N-1}}\frac{u(t,x)-u(t,x+y)}{|y|^{N+\a}} \mu(t,x,|y|, y/|y|) \, dy+ V(t,x)\cdot \n u(t,x) \nonumber
\end{align}
where, for  $(x,r,\th)\in \R^{N-1}\times [0,\infty)\times S^{N-2}$  and $a\in \R^{N-1}$,
$$
\mu(t,x,r,\th):=-\int_{-1}^1\cG'(\varrho p_{u(t,\cdot+a)}(x,x-r\th )+(1-\varrho) p_{u}(x,x-r\th)) Q(\varrho u(\cdot+a)+(1-\varrho)u)(x)  d\varrho
$$
and
$$
V(t,x):=\int_{-1}^1 \frac{\varrho \n u(t,x+a)+(1-\varrho) \n u(t,x) }{ Q(\varrho u(t,\cdot+a)+(1-\varrho) u(t,\cdot))(x)}  H(\varrho u(t,\cdot+a)+(1-\varrho)u(t))(x)  d\varrho.
$$ 
We recall that $\cG'(p)=-(1+p^2)^{-\frac{N+\a}{2}}$.\\
We   now consider $(i)$. Clearly, 
 by Lemma \ref{lemm:SemigropeCorollary}, for all $v\in   C^{1+\b}(\R^{N-1})$, 
\be \label{eq:Latv-nv}
 \|L_a(t) v\|_{ C^{\b-\a}}\leq C   \| \n v\|_{  C^{\b}}  ,
\ee
and 
\be \label{eq:Latsv-nv}
\|[L_a(t)-L_a(s)] v\|_{C^{\b-\a}}\leq C |t-s|^\rho  \| \n v\|_{   C^{\b}}, 
\ee
  with $C$ independent on $a$. Again by  Lemma  \ref{lemm:SemigropeCorollary}, for all $\b'\in (\a,\b)$ and $ a\in \R^{N-1}$,  
$$
L_a(t): {\cC^{1+\b'}_0(\R^{N-1})} \to {\cC^{\b'-\a}_0(\R^{N-1})} 
$$
 is a generator of a strongly continuous analytic semigroup on $ {\cC^{\b'-\a}_0(\R^{N-1})}  $, uniformly in $t\in [0,T]$ and $a$.   

We also observe, from    Proposition \ref{prop:car-inerpol-space},  that for all $a\in \R^{N}$, 
\be\label{eq:calDLh0}
C^{\b'+\rho(1+\a)-\a}(\R^{N-1})\cap {\cC^{\b'-\a}_0(\R^{N-1})} \subset  \calD_{L_a(0)}(\rho,\infty)\qquad\textrm{ for all $\rho\in (0,\frac{1}{1+\a})$.}
\ee
We now prove \eqref{eq:main-2} by induction. Indeed, since   $\rho\leq  \frac{\a}{1+\a}$, we have    $\b'+\rho  (1+\a)-\a\leq  \b'$. \\
We show by induction that for    all   $k\in \N $ and $t\in (0,T]$,  
\be\label{eq:smooth-ut}
u(t)\in \cC^{k+1+\b'}_0(\R^{N-1}) \qquad\textrm{ and } \qquad \| t^{k} \n u(t)\|_{C^{\rho}([0,T],C^{k +\b'})} \leq C_k ,
\ee
for some $C_k=C_k(N,\a,\b,\b',\rho,\nu)>0$.\\
%
Clearly  \eqref{eq:smooth-ut} holds for $k=0$ by the statement of the lemma. We assume that   \eqref{eq:smooth-ut} holds up to order  $k\geq 1$ and we prove the result for $k+1$. 

Consider $L_a(t)$ given by \eqref{eq:defLat}, which we can be written  as  (recall \eqref{eq:def-deeo} and \eqref{eq:decomp-GNMC})
\begin{align*}
L_a(t)u(x)
%
&=\frac{1}{2} \int_0^\infty\int_{S^{N-2}} r^{-2-\a}\d_eu(x,r,\th)(\mu(t,x,r,\th)+\mu(t,x,r,-\th)) \, drd\th \nonumber\\
&+ \int_0^\infty \int_{S^{N-2}} r^{-2-\a}\d_ou(x,r,\th)(\mu(t,x,r,\th)-\mu(t,x,r,-\th))\, drd\th + V(t,x)\cdot \n u(t,x), 
\end{align*}
with $  \mu(t,x,0, \th )=  \mu(t,x,0, -\th )$.
From this  and Lemma \ref{SemigroupLemma001}, we can differentiate \eqref{InitialProblem02-h} $k$ times, so that,   letting $U=\frac{\de ^k}{\de_{x_1\dots x_k}} (\t_h u) $, we then  get 
\be \label{InitialProblem02---ppp}
\begin{array}{rll}
\de_t U + L_h(t) U&=f(t,x) &\hspace{.5cm}\mbox{in }\;\mathbb{R}^{N-1}\times(0,T] 	,
\end{array}
\ee
where 
\begin{align*}
f(t,x)&:=\sum_{S\in \calS_{k-1}}P.V.\int_{\R^{N-1}}\frac{ \de^{|S|} }{\prod_{i\in S}\de x_{i}} \frac{(\t_h u)(x)- (\t_h u)(x-y)}{|y|^{N+\a}} \frac{\de^{k-|S|} }{\prod_{i\not\in S}\de x_{i}}  \mu(t,x,|y|, y/|y|)  \, dy\\
&+\sum_{S\in \calS_{k-1}}  \frac{ \de^{|S|} }{\prod_{i\in S}\de x_{i}}(\t_h \n u)(x) \cdot \frac{\de^{k-|S|} }{\prod_{i\not\in S}\de x_{i}}    V(t,x)
\end{align*}
and  $\calS_{k-1}$ is the set of subsets of $\{1,\dots, k-1\}$.
By \eqref{eq:smooth-ut} and the smoothness of $\cG'$,  for all $S\in \calS_{k-1}$,  
$$
\left\| t^{|S|}  \frac{ \de^{k-|S|}  \mu}{\prod_{i\not\in S}\de x_{i}}  \right\|_{C^{\rho}([0,T], C^{\b'})}\leq C_{k-|S|}
$$
  and   in addition by Lemma \ref{lem:F-smooth-Holder}, 
$$
\left\| t^{|S|+1} \frac{ \de^{|S|}  (\t_h \n u) }{\prod_{i\in S}\de x_{i}}    \right\|_{C^{\rho}([0,T], C^{\b'})}\leq C_{|S|+1}, \qquad \left\|t^{k-|S|}\frac{\de^{k-|S|} }{\prod_{i\not\in S}\de x_{i}}    V \right\|_{C^{\rho}([0,T], C^{\b'-\a})}\leq C_{ k-|S|}.
$$
It then follows, from \eqref{eq:smooth-ut} and  Lemma \ref{SemigroupLemma001}, that  $f(t)\in {\cC^{\b'-\a}_0} (\R^{N-1})$ and  
\be \label{eq:tkp1f}
\left\| t^{k+1} f \right\|_{C^{\rho}([0,T], C^{\b'-\a})}\leq C C_{k}, \qquad t^{k+1} f \big|_{t=0}=0.
\ee
We then define  $v= t^{k+1} \frac{\de ^k (\t_h u)}{\de_{x_1\dots x_k}}=t^{k+1} U\in C^{\rho }([0,T], {\cC^{1+\b'}_0} )$, so that, by \eqref{InitialProblem02---ppp}, 
\be 
\label{InitialProblem022}
\begin{array}{rll}
\de_t v  +L_h(t)  v&=(k+1)t^k  \frac{\de ^k (\t_h u)}{\de_{x_1\dots x_k}}   (t,x)       + t^{k+1} f(t,x) &\qquad\mbox{in }\;\mathbb{R}^{N-1}\times(0,T] 	\\
v(0)&= 0 &\qquad\mbox{in }\;\mathbb{R}^{N-1}\times\{0\}.
\end{array}
\ee
 Letting $g^k:= t^k   \frac{\de ^k(\t_h u)}{\de_{x_1\dots x_k}}  $, then \eqref{eq:smooth-ut} implies that   $g^k\in C^{\rho}([0,T], \cC^{\b'}_0 )$ and  in addition, by  \eqref{eq:calDLh0} and the choice of $\rho\leq  \frac{\a}{1+\a}$,  we find that  
\be\label{eq:dxikthv}
\|g^k \|_{C^{\rho}([0,T], C^{\b'-\a})}+ \|g^k(0)\|_{\calD_{L_0(0)}(\rho,\infty)}\leq \|g^k \|_{C^{\rho}([0,T], C^{\b'})}+C \|g^k(0)\|_{C^{\b'}} \leq C_k.
\ee 
In view of \eqref{eq:Latv-nv} and \eqref{eq:Latsv-nv}, we can thus apply \cite[Proposition 6.1.3]{lunardi2012analytic} to the equation \eqref{InitialProblem022} and use \eqref{eq:dxikthv} together with \eqref{eq:tkp1f} and  deduce that  
\begin{align*}
\|  v\|_{C^{\rho }([0,T], C^{1+\b'})}&\leq C\left(  \|g^k \|_{C^{\rho}([0,T], C^{\b'-\a})}+ \|g^k(0)\|_{\calD_{L_0(0)}(\rho,\infty)}+ \left\| t^{k+1} f \right\|_{C^{\rho}([0,T], C^{\b'-\a})} \right)\\
&\leq C_{k+1} .
\end{align*}
Letting now $h\to 0$, we finally get 
$$
\|  t^{k+1}  \n u \|_{C^{\rho }([0,T], C^{k+1+\b'})}\leq C_{k+1} .
$$
This completes the proof of $(i)$.\\
Finally for $(ii)$,  we use    \eqref{eq:Latv-nv} to  obtain
$$
\| L_h(0) \t_hu_0\|_{\calD_{L_a(0)}(\rho,\infty)}\leq  C  \|\n \t_h u_0\|_{C^{ \b'+\rho(1+\a) }}\leq C  \|\n \t_hu_0\|_{C^{\g_\rho}} .
$$
We can thus apply  \cite[Proposition 6.1.3]{lunardi2012analytic}  to \eqref{InitialProblem02-h}  to get   
$$
\|\t_hu\|_{C^{\rho}([0,T], C^{1+\b'})   }\leq C\left(  \|\t_h u_0\|_{C^{1+\b}}+ \|\n \t_hu_0\|_{C^{\g_\rho}} \right).
$$
Letting $|h|\to 0$ in the above estimate, we find that 
$$
\|\n u\|_{C^{\rho}([0,T], C^{1+\b'})   } \leq  C  \|\n u_0\|_{C^{1+{\g_\rho}}}  .
$$
That is  \eqref{eq:main-3}.
	\end{proof}

\begin{proof}[Proof of Theorem \ref{MainResult1} and  Theorem \ref{MainResult1-part-2}  ]
For $n\in \N$, we let  $\eta_n\in C^\infty_c(\R^{N-1})$   such that $\eta\equiv 1$ for $|x|\leq n$, 	 $\eta\equiv 0$ for $|x|\geq 2n$	 and $|D^k \eta_n|\leq\frac{1}{n^k}$ on $\R^{N-1}$.
%
It then follows  that $\eta_n u_0\in \cC^{1+\b }_0 (\R^{N-1})$  thanks to  Proposition \ref{prop:car-C_0} and  that
\be \label{eq:etanu0n}
\|\n (\eta_n u_0)\|_{C^{\g_\rho} }\leq 2 \|\n   u_0\|_{C^{\g_\rho} }\leq 2\nu,
\ee
where $\g_\rho=\b+\rho (1+\a)$.
By Theorem \ref{MainResult}  there exists a unique function   $u^n\in  E_T	$ solving
\begin{align}
\label{eq:nmc-flow-u-n}
	\left\{
	\begin{array}{rll}
	\de_t u^n+\cH(u^n)&=0&\hspace{.5cm}\mbox{in }\;[0,T]\times\mathbb{R}^{N-1}	\\
	u^n(0)&=\eta_n u_0&\hspace{.5cm}\mbox{in }\;\mathbb{R}^{N-1}
	\end{array}
	\right.
	\end{align}
 and 	satisfying 
\be \label{eq:norminET-cut}
	 	  \| u^n - \eta_n u_0 \|_{E_T}  \leq C_0(N,\a,\b,\g,\rho,\nu)  ,
\ee
%
%
%
%
%

Since, by \eqref{eq:etanu0n},  the sequence   $(\n (\eta_n u_0))_n$	 is bounded in $ C^{\b}(\R^{N-1})$,  it then follows, from \eqref{eq:norminET-cut} and  the   Arzel\`a-Ascoli theorem that, for $\rho'<\rho$ and $\b'<\b$, up to a subsequence, the sequence $(u^n)_n$ converges in $C^{\rho'}([0,T],C^{1+\b'}_{loc}(\R^{N-1}))\cap C^{1+\rho'}([0,T],C^{\b'-\a}_{loc}(\R^{N-1}))$ to some function  $u\in C^{\rho}([0,T],C^{1+\b}_{loc}(\R^{N-1}))\cap C^{1+\rho}([0,T],C^{\b-\a}_{loc} (\R^{N-1}))$. In particular $\cH(u^n) \to \cH(u)$  pointwise locally uniformly.   Hence passing to the limit in \eqref{eq:nmc-flow-u-n}, we get
\begin{align}
\label{eq:nmc-flow-u-n-zz}
	\left\{
	\begin{array}{rll}
	\de_t u+{\sqrt{1+|\nabla u|^2}} H(u)&=0&\hspace{.5cm}\mbox{in }\;[0,T]	\times\mathbb{R}^{N-1}\\
	u(0)&=  u_0&\hspace{.5cm}\mbox{in }\;\mathbb{R}^{N-1}.
	\end{array}
	\right.
	\end{align}
Letting $n\to \infty$ in \eqref{eq:norminET-cut}, we then see that  
$   \| u-u_0 \|_{C^{\rho}([0,T],C^{ \b} )\cap C^{1+\rho}([0,T],C^{ \b-\a} )}\leq   C_0'' .$ \\
  
For the uniqueness, we consider $v\in C^{\rho}([0,T],C^{1+\b}_{loc}(\R^{N-1}))\cap C^{1+\rho}([0,T],C^{\b-\a}_{loc}(\R^{N-1}))$ satisfying  \eqref{eq:nmc-flow-u-n-zz} 
	and  $   v-u_0 \in {E_T} $.  Then $w=u-v\in E_T$ satisfies
	 \begin{align}\label{eq:equ--solBw-mmp}
	\left\{
	\begin{array}{rll}
	\de_t w+B[u(t),v(t)]w&=0&\hspace{.5cm}\mbox{in }\;[0,T]	\times\mathbb{R}^{N-1}\\
	w(0)&= 0&\hspace{.5cm}\mbox{in }\;\mathbb{R}^{N-1},
	\end{array}
	\right.
	\end{align}
	where, $B[u,v]w=\int_0^1D\cH(\tau u+(1-\tau)v)[w]\, d \tau$.
	Hence by Corollary  \ref{cor:mmp-lnmc}, we deduce that $w=0$.\\
	Now \eqref{eq:main-3-cor}    is a consequence of   \eqref{eq:main-3}, by compactness and thus  the proof of Theorem \ref{MainResult1} is thus complete. \\
Using  \eqref{eq:main-2}, we obtain  \eqref{eq:main-2-cor} thanks to the   Arzel\`a-Ascoli theorem, which is Theorem \ref{MainResult1-part-2}.\\

\end{proof}		 

 Our next result shows that the (global) sign of the initial condition $u_0$ and the one of the the initial fractional mean curvature $H(u_0)$ are preserved.  It also completes the proof of  Theorem  \ref{th:mt-univ-est}.
\begin{lemma}\label{lem:mmp-u-sol-zero}
Under the assumptions of  Theorem \ref{MainResult1}, we have the following results.
\begin{enumerate}
\item[(i)]  $\|\n u\|_{L^\infty((0,T)\times \R^{N-1})}\leq \|\n u_0\|_{L^\infty}$,
\item[(ii)]  $\displaystyle\sup_{[0,T]\times \R^{N-1}} \de_t u=-\inf_{ \R^{N-1}} \cH(u_0)\qquad$ and $\qquad\displaystyle\inf_{[0,T]\times \R^{N-1}} \de_t u=-\sup_{ \R^{N-1}} \cH(u_0)$.

\item[(iii)] If moreover $u_0\in L^\infty(\R^{N-1})$ then  
$$
\displaystyle\sup_{[0,T]\times \R^{N-1}} u=\sup_{ \R^{N-1}} u_0\qquad \textrm{ and }\qquad\displaystyle\inf_{[0,T]\times \R^{N-1}} u=\inf_{ \R^{N-1}} u_0
$$
\end{enumerate}
\end{lemma}
\begin{proof}
We start with $(iii)$. If $u_0\in L^\infty(\R^{N-1})$ then by   \eqref{eq:main-1-cor}  we have $u\in C^{\rho}([0,T],C^{1+\b} )\cap C^{1+\rho}([0,T], C^{\b-\a} )$. Now,  in view of \eqref{eq:exp-nmc-Bu}, we can apply Proposition \ref{prop:mmp}  to $u$ and $-u$ to  get $(iii)$. \\  
Next to prove $(i)$,  we define   $\t_hu(t,x)=\frac{u(t,x+h)-u(t,x)}{|h|}\in C([0,T], C^{1+\b})$, for $0<|h|<1$. We then have  
$$
\de_t\t_hu+ B[u(t,\cdot+h),u(t)]\t_hu(t)=0\qquad\textrm{ in $ [0,T]\times \R^{N-1}$}.
$$
It thus follows from  Corollary  \ref{cor:mmp-lnmc}   that $\|\t_h u\|_{L^\infty((0,T)\times \R^{N-1})}\leq \|\t_h u_0\|_{L^\infty}$. Hence, letting $h\to 0$, we get  $(i)$.\\
 Finally, by Theorem \ref{MainResult1-part-2} and Lemma \ref{lem:F-smooth-Holder} we have $\de_t u=-\cH(u)\in C^\rho([\e,T], C^{k+\b'-\a} )$ for all $\e\in (0,T)$ and $k\geq 0$.  Letting  $\t\in (\e,T)$, we define $u^\t(t,x)=\frac{u(t+\t,x)-u(t,x)}{\t}  $, for $\e<t\leq T-\t$. We thus get $ u^\t\in  C^{\rho}([\e,T-\t],C^{1+\b} )\cap C^{1+\rho}([\e,T-\t], C^{\b-\a} )$  and  
$$
\de_t u^\t +B[u(t+\t),u(t)]u^\t(t)=0\qquad\textrm{ on $ [\e,T-\t]\times \R^{N-1}$}.
$$
Applying  Corollary  \ref{cor:mmp-lnmc}, we   obtain
$$
 \sup_{[\e,T] \times \mathbb{R}^{N-1} } u^\t=   \sup_{ x\in \mathbb{R}^{N-1} }u^\t(\e,x)\qquad\textrm{ and } \qquad
 \inf_{[\e,T] \times \mathbb{R}^{N-1} } u^\t=   \inf_{ x\in \mathbb{R}^{N-1} }u^\t(\e,x) .
$$
  Since $\de_t u=-\cH(u)\in C^\rho( [0,T] , L^\infty ) $,  letting first $\t\to 0$ and then  $\e\to 0$ in the above identities, we get $(ii)$.
%
%
	\end{proof}
 \begin{proof}[Proof of Theorem  \ref{th:mt-univ-est} (completed)]
It suffices to apply Lemma \ref{lem:mmp-u-sol-zero}. 

	\end{proof}

\end{document}